\documentclass[12pt, openany]{amsart}
\usepackage[utf8]{inputenc}
\usepackage[T1]{fontenc}
\usepackage[pdftex]{graphicx}
\usepackage{amsmath,amssymb,amsthm}
\usepackage{mathrsfs}
\usepackage{hyperref}
\usepackage{pgf,tikz,pgfplots}
\usepackage{tikz}
\usepackage{pgfplots}
\pgfplotsset{compat=1.15}
\usetikzlibrary{arrows}
\usetikzlibrary{calc,angles,quotes}
\usetikzlibrary{arrows}
\usepackage{graphicx}
\usepackage{array}
\usepackage{color}
\usepackage{mathtools}

\newtheorem{theorem}{Theorem}
\theoremstyle{definition}
\newtheorem{definition}[theorem]{Definition}
\theoremstyle{remark}

\theoremstyle{remark}
\newtheorem{remark}[theorem]{Remark}
\theoremstyle{theorem}

\theoremstyle{theorem}

\theoremstyle{theorem}
\newtheorem{proposition}[theorem]{Proposition}
\theoremstyle{theorem}
\newtheorem{corollary}[theorem]{Corollary}
\newtheorem*{mainthm}{Main Theorem}

\newcommand{\bigo}[1]{\mathcal{O}\left(#1\right)}

\newcommand{\ZZ}{\mathbb{Z}}

\newcommand{\RR}{\mathbb{R}}

\renewcommand{\SS}{\mathbb{S}}

\newcommand{\TT}{\mathbb{T}}

\definecolor{deeppink}{rgb}{1,0.078,0.576}

\begin{document}

\title[]{A billiard table close to an ellipse is deformationally spectrally rigid among dihedrally symmetric domains}

\author{Corentin Fierobe}
\address{Department of Mathematics, University of Rome Tor Vergata, Via della Ricerca Scientifica 1, 00133 Rome, Italy}
\email{cpef@gmx.de}

\author{Vadim Kaloshin}
\address{Institute of Science and Technology Austria, Am Campus 1, 3400 Klosterneuburg, Austria}
\email{vadim.kaloshin@gmail.com}

\author{Alfonso Sorrentino}
\address{Department of Mathematics, University of Rome Tor Vergata, Via della Ricerca Scientifica 1, 00133 Rome, Italy}
\email{sorrentino@mat.uniroma2.it}

\begin{abstract} 
We prove that a a strongly convex planar domain (Birkhoff table) with dihedral symmetry, which is sufficiently close in a finitely smooth  topology to an ellipse, is deformationally spectrally rigid within the class of domains preserving this symmetry. More precisely, any smooth one-parameter family of such domains that preserves the length spectrum (i.e., the set of lengths of periodic billiard orbits) must consist only of rigid motions of the initial domain. The proof combines two types of dynamical data: the asymptotic behavior of certain symmetric periodic orbits, as previously used in the rigidity of nearly circular domains, and new spectral information derived from KAM invariant curves, obtained from Mather's beta function and its derivatives (in the Whitney sense) at some suitable rotation numbers.
\end{abstract}

\maketitle


\section{Introduction}

A \emph{Birkhoff billiard} describes the motion of a point particle moving with unit speed inside a strongly convex domain $\Omega \subset \mathbb{R}^2$ with smooth boundary. More precisely, we assume $\partial \Omega$ is a $\mathscr{C}^r$ curve with $r \geq 3$ and that the curvature is strictly positive at every point of the boundary; we call such a domain $\Omega$ a \emph{Birkhoff table}. The particle follows straight-line motion in the interior and reflects elastically at the boundary, with the angle of incidence equal to the angle of reflection.

The dynamics are encoded by the \emph{billiard map} $f_\Omega: M \to M$, where the phase space $M$ consists of inward-pointing unit vectors based at the boundary. For $(x,v) \in M$, the map sends $(x,v)$ to $(x',v')$, where $x'$ is the next collision point and $v'$ is the reflected velocity (see Figure~\ref{figurebilliard}). We refer to \cite{Siburg, Tabach} for a comprehensive introduction to billiards.

\begin{figure}\label{figurebilliard}
\definecolor{qqwuqq}{rgb}{0,0.39215686274509803,0}
\begin{tikzpicture}[line cap=round,line join=round,>=triangle 45,x=2cm,y=2cm]
\clip(-1.8,-1) rectangle (1.8,1);
\draw [shift={(-0.8742616278285551,0.7057918824458678)},line width=0.2pt,color=qqwuqq,fill=qqwuqq,fill opacity=0.10000000149011612] (0,0) -- (-9.235736422648966:0.13014296530989553) arc (-9.235736422648966:29.63185395731343:0.13014296530989553) -- cycle;
\draw [shift={(1.2516955384490343,0.3601010365975154)},line width=0.2pt,color=qqwuqq,fill=qqwuqq,fill opacity=0.10000000149011612] (0,0) -- (-99.07534275578051:0.13014296530989553) arc (-99.07534275578051:-57.93294158117337:0.13014296530989553) -- cycle;
\draw [rotate around={0:(0,0)},line width=1pt] (0,0) ellipse (2.7196525cm and 1.842962213cm);
\draw [line width=0.5pt] (-0.8742616278285551,0.7057918824458678)-- (1.2516955384490343,0.3601010365975154);
\draw [-latex,line width=0.3pt] (-0.8742616278285551,0.7057918824458678) -- (-0.42881442284704246,0.6333600275085152);
\draw [-latex,line width=0.3pt] (1.2516955384490343,0.3601010365975154) -- (1.1856741508046915,-0.05322317657592285);
\begin{scriptsize}
\draw[color=black] (0.19125759568047124,0.35) node {$L(s,s')$};
\draw[color=black] (-0.4,-0.7) node {$\Omega$};
\draw[color=black] (-1,0.75) node {$s$};
\draw[color=black] (-0.6,0.75) node {$\varphi$};
\draw[color=black] (-0.7,0.55) node {$v$};
\draw[color=black] (1.4,0.4) node {$s'$};
\draw[color=black] (1.28,0.05) node {\tiny $\varphi'$};
\draw[color=black] (1.1,0.2) node {$v'$};
\end{scriptsize}
\end{tikzpicture}
\caption{Two successive billiard impact points in a strictly convex billiard domain $\Omega$. Here $f(s,v)=(s',v')$ and $L(s,s')$ measures the distance between the two impact points. } 
\end{figure}
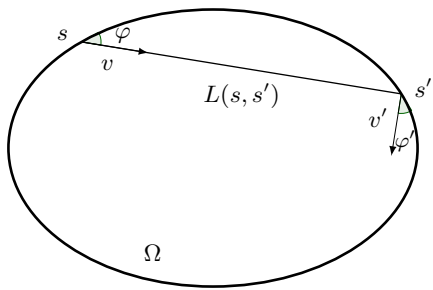

Despite the simple mechanical law, billiard dynamics are remarkably complex and intimately connected to the geometry of the domain. While the shape of the domain completely determines the billiard dynamics, a more subtle and intriguing question is the following: \textit{To what extent can dynamical information be used to reconstruct the shape of a billiard table?}

\begin{remark}\label{remaction}
    Orbits of the billiard map admit a natural variational formulation as critical points of the length functional. In fact, the billiard map is a symplectic twist map whose generating function is the distance between boundary points (see \cite{Siburg, Tabach}), thus the length plays the role of the \emph{action}.
\end{remark}

Therefore, it is natural to study whether the lengths of periodic orbits encode sufficient information to determine the shape of the table. In this article, we focus on this inverse problem, related to the so-called \emph{length spectrum}.

\subsection{Periodic Orbits, Length Spectrum, and Spectral Rigidity}

Periodic orbits have been central to billiard theory since Birkhoff's seminal work \cite{Birkhoff}. Geometrically, each periodic trajectory corresponds to a closed (not necessarily convex) polygon inscribed in $\Omega$, whose segments satisfy the reflection law at each vertex.

For a periodic orbit with $q$ reflections that winds $p$ times around the boundary, the \emph{rotation number} is given by $p/q \in (0, 1/2] \cap \mathbb{Q}$. Birkhoff \cite{Birkhoff} proved that for each $p/q$ in lowest terms, there exist at least two periodic orbits with that rotation number.

\begin{remark}
    Reversing the direction of a periodic billiard trajectory with rotation number $p/q \in (0, 1/2] \cap \mathbb{Q}$ yields a periodic trajectory with rotation number $(q-p)/q \in [1/2,1)$. Hence, it suffices to consider rotation numbers in $(0, 1/2] \cap \mathbb{Q}$.
\end{remark}

\begin{remark}
    In terms of the variational characterization from Remark~\ref{remaction}, for any rotation number $p/q \in (0,1/2] \cap \mathbb{Q}$, one of the periodic orbits in Birkhoff's result is obtained by maximizing the perimeter of inscribed polygons with that rotation number, while the other is obtained via a min-max argument.
\end{remark}

The \emph{length spectrum} of $\Omega$ is defined as:
\[
\mathscr{L}(\Omega) := \mathbb{N}^+ \cdot \{\text{lengths of periodic orbits}\} \cup \mathbb{N}^+ \cdot \ell(\partial\Omega).
\]

A fundamental inverse problem consists in determining how much $\mathscr{L}(\Omega)$ reveals about the geometry of the billiard table and, ideally, whether it determines $\Omega$ up to isometries.

\medskip
\textbf{Question 1 (Global Rigidity).} \textit{If two Birkhoff billiard tables have the same length spectrum, are they necessarily congruent up to rigid motions?}

\begin{remark}
    A deep connection exists between the length spectrum and the Laplace spectrum $\mathrm{Spec}_\Delta(\Omega)$ on $\Omega$ with Dirichlet boundary conditions. Anderson and Melrose \cite{AM} (see also \cite{GM, PS, SafVas}) established that the wave trace
    \[
    w(t) := \mathrm{Re}\left( \sum_{\lambda \in \mathrm{Spec}_\Delta(\Omega)} e^{i\sqrt{-\lambda} t} \right)
    \]
    is well-defined as a distribution and is smooth away from the length spectrum. Specifically, its singular support is contained in $\pm \mathscr{L}(\Omega) \cup \{0\}$, and generically they coincide. This relates directly to Kac's famous question \cite{Kac} ``{\it Can one hear the shape of a drum?}~''. While answered negatively for general manifolds \cite{GordonWebbWolpert}, this question remains open for strongly convex planar smooth domains. Melrose \cite{MelroseIsospectral} and Osgood, Phillips and Sarnak \cite{OsgoodPhillipsSarnak1, OsgoodPhillipsSarnak2, OsgoodPhillipsSarnak3} demonstrated that Laplace-isospectral sets of planar domains are compact in the $\mathscr{C}^{\infty}$ topology.  Vig proved an analogous result for the marked length spectrum \cite{Vig}.
\end{remark}

One can also consider a ``local'' version of Question~1 by studying one-parameter families of billiard tables.

\begin{definition}
    A \emph{one-parameter family} of Birkhoff tables is a family $(\Omega_{\tau})_{\tau \in I}$, where $I \subset \mathbb{R}$ is an interval containing $0$, each $\Omega_{\tau}$ is a Birkhoff table, and the boundary $\partial\Omega_{\tau}$ depends smoothly ($\mathscr{C}^r$, $r \in \mathbb{N}_{>0} \cup \{\infty, \omega\}$) on the parameter $\tau$.
\end{definition}

Given such a family, one can study how the lengths of periodic orbits vary under infinitesimal boundary deformations.

\medskip
\textbf{Question 2 (Deformational Rigidity).} \textit{Let $(\Omega_{\tau})_{\tau \in I}$ be a one-parameter family of Birkhoff tables such that $\mathscr{L}(\Omega_\tau) = \mathscr{L}(\Omega_0)$ for every $\tau \in I$. Is it true that each $\Omega_{\tau}$ is obtained from $\Omega_0$ by a rigid motion?}

A domain $\Omega$ is said to be \emph{rigid} if Question~2 holds for any one-parameter family $(\Omega_{\tau})_{\tau \in I}$ with $\Omega_0 = \Omega$. One can impose further conditions on the deformations; in our case, we require the domains $\Omega_{\tau}$ to have \emph{dihedral symmetry}, {\it i.e.}, to be both axis-symmetric and centrally symmetric.

\subsection{Main Result}

In this article, we provide a partial answer to Question~2 for domains near an ellipse and under symmetry assumptions.

\begin{mainthm}\label{theorem:main}
    Let $\mathscr{E}$ be an ellipse. There exists an integer $r = r(\mathscr{E}) > 0$ and $\varepsilon = \varepsilon(\mathscr{E}) > 0$ such that the following holds. Let $\Omega$ be a Birkhoff table with $\mathscr{C}^{r}$-smooth boundary and dihedral symmetry, which is $\varepsilon$-$\mathscr{C}^r$-close to $\mathscr{E}$. Then, $\Omega$ is rigid under one-parameter smooth deformations within the class of dihedrally symmetric $\mathscr{C}^r$-Birkhoff tables.
\end{mainthm}

\medskip
\begin{remark}
    \textbf{(i)} As will be clear from the proof, $r$ and $\varepsilon$ depend only on the eccentricity of $\mathscr{E}$. \\
    \textbf{(ii)} The $\mathscr{C}^r$-topology on Birkhoff tables (for integer $r > 0$) corresponds to the Whitney $\mathscr{C}^r$-topology on the set of $\mathscr{C}^r$-smooth embeddings $\gamma: \mathbb{S}^1 \to \mathbb{R}^2$ whose image $\gamma(\mathbb{S}^1)$ is the boundary of a Birkhoff table. Two tables $\Omega, \Omega'$ are $\varepsilon$-$\mathscr{C}^r$-close if $\|\partial\Omega - \partial\Omega'\|_{\mathscr{C}^r} \leq \varepsilon$.
\end{remark}

\subsection{Comparison with Previous Literature}

The length spectral rigidity of Birkhoff tables is a central open question in the field. The result most closely related to our main theorem is by De Simoi, Kaloshin, and Wei \cite{DKW}, who proved the following: Let $\mathcal{M}$ be the class of Birkhoff tables with  $\mathscr{C}^9$-smooth boundary and axial symmetry, which are sufficiently close (in the $\mathscr{C}^1$-norm) to a circle. Then, any $\Omega \in \mathcal{M}$ is deformationally spectrally rigid in $\mathcal{M}$, meaning that any $\mathscr{C}^1$-smooth one-parameter length-isospectral family in $\mathcal{M}$ is necessarily an isometric family. 

Other related results focus on another set of information associated to periodic orbits, that is the so-called \emph{Marked length spectrum},
which associates to each rational rotation number $p/q \in (0,1/2] \cap \mathbb{Q}$ the maximal length of periodic orbits with that rotation number. The marked length spectrum is equivalent  to \emph{Mather's $\beta$-function} of the billiard map, a central object in the study of symplectic twist maps (see \cite{Siburg} and also Section \ref{secbetafunction}).

One can also consider the analogous rigidity question for the {marked length spectrum}: if two domains share the same marked length spectrum, must they be isometric? Similarly, one can study one-parameter families of domains with constant marked length spectrum. Note that any deformation preserving the length spectrum automatically preserves the marked length spectrum as well (see \cite{Siburg}).

Results in this direction include:
\begin{itemize}
    \item[{(i)}] If $\Omega$ has the same marked length spectrum as a disk, then it is a disk \cite{Siburg}. Another proof uses Taylor coefficients of Mather's $\beta$-function at $0$ (see \cite{MM, SorDCDS}). Recent work \cite{BBS} discusses how a single value of Mather’s $\beta$-function can determine whether a Birkhoff table is circular.
    
    \smallskip
    \item[(ii)] For ellipses, the situation is less clear. \cite{SorDCDS} showed that the Taylor coefficients of the associated Mather's $\beta$-function at $0$ determine an ellipse within the family of ellipses. \cite{Bialyellipses} computed an explicit expression for Mather's $\beta$-function for elliptic billiards, proving that the value of $\beta$ at $1/4$ and its derivative at $0$ also determine an ellipse uniquely (alternatively, one can consider the values of $\beta$ at $\frac 12$ and any other rational in $(0,\frac 12)$)
    
    \smallskip
    \item[{(iii)}] Results on the marked length spectral rigidity of elliptic billiards are connected to the integrability of the billiard map and the Birkhoff conjecture, which states that the only integrable Birkhoff tables are ellipses and disks. Mather \cite{Mather90} showed that the $\beta$ function is differentiable at a rational $\rho$ if and only if there exists an invariant curve of periodic orbits with rotation number $\rho$. Thus, $\mathscr{C}^1$ regularity of $\beta$ on an interval implies $\mathscr{C}^0$-integrability. This translates recent results on the Birkhoff conjecture into rigidity properties of the associated $\beta$-function.
    
\noindent     More specifically:
    \begin{itemize}
        \item[a)] \textit{(Centrally symmetric case)}: Results in \cite{BialyMironov} imply that if $\Omega$ is centrally symmetric and $\beta$ is differentiable on $(0,1/4]$, then $\Omega$ is an ellipse.
        \item[b)] \textit{(Perturbative case)}: Results in \cite{ADK, KS} imply that if $\Omega$ is $\mathscr{C}^1$-close to an ellipse and $\beta$ is differentiable at all rationals $1/q$ with $q \geq 3$, then $\Omega$ is an ellipse. This can be generalized using \cite{Koval} (see also \cite{HKS}) to consider integrability near the boundary.
    \end{itemize}
\end{itemize}

Other related results for non-Birkhoff billiards include:
\begin{itemize}
    \item Under symmetry and genericity assumptions, \cite{DKL} proved that the marked length spectrum determines the geometry of billiard tables obtained by removing finitely many strictly convex analytic obstacles satisfying the non-eclipse condition.
    \item For Bunimovich stadia and squash-type stadia, \cite{CKZ} established dynamical spectral rigidity under additional symmetry assumptions.
\end{itemize}

\begin{remark}
In a related direction, in \cite{HKS_Duke} the authors showed that for a generic domain, one can recover from the marked length spectrum of the domain the eigendata corresponding to Aubry-Mather periodic orbits ({\it i.e.}, periodic orbits of maximal length among all orbits with the same rotation number).
\end{remark}

\medskip

\subsubsection{Other related works}
Another line of investigation concerns \textit{Laplace spectral rigidity} and Kac's question. For comprehensive overviews, see \cite{HezariZelditch} and the survey \cite{ZelditchSurvey15}. A notable result appears in \cite{HZ2}, where the authors prove that ellipses of sufficiently small eccentricity are Laplace spectrally unique among all smooth domains. For one-parameter families, \cite{HezariZelditch} gave a positive answer for analytic Laplace isospectral deformations of ellipses preserving biaxial reflectional symmetries (see also Popov and Topalov \cite{PopovTopalov1}).
Koval and Vig \cite{KovalVig} showed however that for a large set of domains the Laplace spectrum does not determine the length spectrum.

Another related problem concerns \textit{isospectral deformations of manifolds} without boundary, where the length spectrum consists of the lengths of closed geodesics. The corresponding inverse problem has been studied in various contexts. A smooth family of metrics $(g_\tau)_{|\tau|\leq1}$ on a compact boundaryless Riemannian manifold $(M,g)$ is an \textit{isospectral deformation} if the length spectrum remains constant. The manifold is \textit{spectrally rigid} if it admits no non-trivial isospectral deformations.
Guillemin--Kazhdan~\cite{GuilleminKazhdan} proved spectral rigidity for negatively curved surfaces, a result later extended to higher dimensions~\cite{CS}.

Similarly, one can consider the \textit{marked length spectrum} of a Riemannian manifold, where the role of the rotation number is played by the homotopy class of the closed geodesics. For negatively curved manifolds, the marked length spectrum uniquely determines the metric~\cite{Croke,Otal}. Recently, Guillarmou and Lefeuvre \cite{GL} proved that in all dimensions, the marked length spectrum of a Riemannian manifold $(M,g)$ with Anosov geodesic flow and non-positive curvature locally determines the metric: two sufficiently close metrics with the same marked length spectrum must be isometric.\\
In a similar context \cite{AgapovBialyMironov}, the authors proved the existence of integrable deformations of the Liouville metric on $\TT^2$ by Riemannian metric endowed with a small magnetic potential.

In symplectic dynamics, De la Llave, Marco and Moriy\'on~\cite{dMM} showed that Anosov symplectomorphisms admit no non-trivial deformations preserving the action spectrum, with hyperbolicity and density of periodic points being key factors.

\subsection{Organization of the Article}
The article is organized as follows:
\begin{itemize}
\item Section~\ref{section:outline} provides an outline of {the key idea and the strategy of }the proof of the Main Theorem.
\item In Section~\ref{section:tools_isospectral} we introduce the tools used for studying isospectral deformations: the $\tau$-derivative of Mather's $\beta$-function (see {Subsection~\ref{subsection:tau_variation_beta}}) and the $\tau$-derivative of the perimeters of symmetric periodic orbits (see {Subsection~\ref{subsection:DKW_proof}}). 
\item In {Section~\ref{section:isospectral_operator}} we define  the  isospectral operators, whose invertibility implies the rigidity. 
\item The proof of the Main Theorem is contained in {Section~\ref{section:proof_main}}.
\item Appendix \ref{section:sinus_family} shows that the family $\sin^{2j}(2\pi\theta)$, $j\geq 0$, is equivalent to the family $\cos(4\pi j\theta)$, $j\geq 0$, in $L^2(\SS^1)$.
\item Appendix \ref{section:mobius_operator} introduces two elementary operators on a subspace of $L^2(\SS^1)$ that are one the inverse of the other.
\item The $\tau$-derivative of Mather's $\beta$ function is computed explicitly in the case of ellipses in Appendix \ref{section:ellipses_total}.
\item Appendix \ref{section:operator_completion} provides conditions to deduce the invertibility of the so-called \textit{completions of operators}.
\end{itemize}

\medskip

\subsection{Acknowledgments} 
CF and VK acknowledge the support of the ERC Advanced Grant SPERIG (nr. 885707).
CF and AS  acknowledge the support of the Italian Ministry of University and Research’s  PRIN 2022 grant ``{\it Stability in Hamiltonian dynamics and beyond}’', as well as the Department of Excellence grant MatMod@TOV (2023-27) awarded to the Department of Mathematics of University of Rome Tor Vergata. 
AS is member of the INdAM research group GNAMPA and the UMI group DinAmicI.

\usetikzlibrary{decorations.markings}
\usetikzlibrary{decorations.pathreplacing}
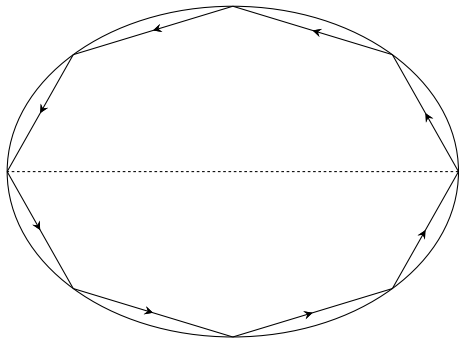
\begin{figure}[h!]
	\label{figure:symmetric_orbit}
	\begin{tikzpicture}[xscale=0.75, yscale =1.1, decoration={
			markings,
			mark=at position 0.5 with {\arrow{>}}}
		] 
		\def\axisX{4} 
		\def\axisY{2} 
		
		\draw[-latex] (0,0) ellipse (\axisX cm and \axisY cm);
		
		\foreach \i in {1,...,8} {
			\coordinate (P\i) at ({\axisX * cos(45 * (\i - 1))}, {\axisY * sin(45 * (\i - 1))});
		}
		
		\draw[-, >=stealth, decoration={markings, mark=at position 0.5 with {\arrow{>}}}, postaction={decorate}] (P1) -- (P2);
		
		\draw[-, >=stealth, decoration={markings, mark=at position 0.5 with {\arrow{>}}}, postaction={decorate}] (P2) -- (P3);
		
		\draw[-, >=stealth, decoration={markings, mark=at position 0.5 with {\arrow{>}}}, postaction={decorate}] (P3) -- (P4);
		
		\draw[-, >=stealth, decoration={markings, mark=at position 0.5 with {\arrow{>}}}, postaction={decorate}] (P4) -- (P5);
		
		\draw[-, >=stealth, decoration={markings, mark=at position 0.5 with {\arrow{>}}}, postaction={decorate}] (P5) -- (P6);
		
		\draw[-, >=stealth, decoration={markings, mark=at position 0.5 with {\arrow{>}}}, postaction={decorate}] (P6) -- (P7);
		
		\draw[-, >=stealth, decoration={markings, mark=at position 0.5 with {\arrow{>}}}, postaction={decorate}] (P7) -- (P8);
		
		\draw[-, >=stealth, decoration={markings, mark=at position 0.5 with {\arrow{>}}}, postaction={decorate}] (P8) -- (P1);
		
		\draw[dash pattern=on 1pt off 1pt] (P1) -- (P5);
        
	\end{tikzpicture}
	\label{figure:AxiallySymmetricOctagon}
	\caption{An axially symmetric billiard orbit of rotation number $\omega = 1/8$.}
\end{figure}

\section{Outline of the proof of the Main Theorem}\label{section:outline}

The proof of the Main Theorem uses two kinds of dynamical information extracted from the length spectrum of a $\mathscr C^r$ Birkhoff billiard table~$\Omega$, which are related to:

\begin{enumerate}
\item[(D1)] Periodic billiard orbits of rotation number $1/q$ (convex inscribed $q$-gons) for $q\geq 3$, as in \cite{DKW}; see Fig.~\ref{figure:AxiallySymmetricOctagon}.
\item[(D2)] Rotational invariant curves with Diophantine rotation number, whose existence is guaranteed by KAM theory for sufficiently smooth boundaries \cite{Lazutkin}.
\end{enumerate}

A smooth isospectral deformation $(\Omega_\tau)$ can be encoded by an even deformation map $n_\tau\in H$, where $H$ is a Sobolev subspace of $L^2(\mathbb S^1)$; the deformation is trivial exactly when $n_\tau\equiv 0$. Under dihedral symmetry, $n_\tau$ is $1/2$-periodic. To prove the Main Theorem, it therefore suffices to show that any isospectral deformation forces its associated map $n$ to vanish.

The main step is to construct a linear operator
\[
T_\Omega : H_{1/2} \to h,
\]
with $H_{1/2}\subset H$ the subspace of $1/2$-periodic functions and $h$ another Sobolev-type sequence space, such that
\[
\text{isospectrality} \quad\Longrightarrow\quad T_\Omega(n)=0.
\]

\medskip

\noindent\textbf{Step 1 (using periodic orbits).}  
From the family of periodic trajectories (D1), one obtains a collection of linear conditions $\ell_q(n)=0$ that all isospectral deformations must satisfy. These define an operator $S$ whose kernel is finite-dimensional and spanned by a small number of low Fourier modes.

\medskip

\noindent\textbf{Step 2 (using KAM curves).}  
The invariant curves (D2) provide additional linear conditions, obtained from the ``KAM density''
$
\mu_\Omega(\omega,x),
$
which satisfy
\[
\int_0^1 n(x)\,\partial_\omega^j \mu_\Omega(\omega,x)\,dx=0,\quad j\geq 0.
\]
For ellipses, the resulting family of functionals is large enough to eliminate all remaining directions in $\ker S$, and this property persists continuously for nearby billiard tables.

\medskip

\noindent\textbf{Conclusion.} Combining the constraints from (D1) and (D2) yields an invertible operator $T_{\Omega}$ extending $S$. Hence if an isospectral deformation satisfies $T_\Omega(n)=0$ then $n=0$. Thus every isospectral deformation is trivial, proving the Main Theorem.

\section{Isospectral deformations with dihedral symmetry}
\label{section:tools_isospectral}

Let $(\Omega_{\tau})_{\tau\in I}$ be a one-parameter family of Birkhoff tables. For any given $\tau\in I$, consider an $|\partial\Omega_{\tau}|$-periodic map $\gamma_{\tau}:\RR\to\RR^2$ parametrizing the boundary of $\partial\Omega_{\tau}$ by arc-length.

In terms of regularity, given an integer $r>0$ we say that the deformation $(\Omega_{\tau})_{\tau\in I}$ is $\mathscr C^r$-smooth if $(\tau,s)\mapsto\gamma_{\tau}(s)$ is a $\mathscr C^r$-smooth map.

\begin{definition}
The familiy $(\Omega_{\tau})_{\tau\in I}$  is said to be \textit{isospectral} if for any $\tau\in I$,
\[
\mathscr L(\Omega_{\tau}) = \mathscr L(\Omega_{0}).
\]
\end{definition}

It is easy to deduce that if the familiy $(\Omega_{\tau})_{\tau\in I}$  is isospectral then all the domains $\Omega_{\tau}$ 
have the same perimeter (see \cite[Proposition 3.2.8]{Siburg}).

\subsection{Mather's $\beta$-function}\label{secbetafunction}
 
Let $\Omega \subset \mathbb{R}^2$ be a Birkhoff table parametrized by arc-length by the function $s\mapsto\gamma(s)$, and let 
\[
f : M \to M, \qquad M = \RR/|\partial\Omega|\ZZ\times (0,\pi),
\]
denote the associated billiard map. 
The billiard dynamics admits a variational formulation: if $(s_i)_{i \in \mathbb{Z}}$ is a bi-infinite sequence 
of boundary points, its associated \textit{action} is 
\[
A((s_k)) =- \sum_{k \in \mathbb{Z}} |\gamma(s_{k+1}) - \gamma(s_k)|,
\]
given by the sum of chord lengths. A configuration $(s_i)$ is called \emph{minimizing} if every finite segment 
$(s_m, \dots, s_n)$, with $m\leq n$, minimizes the \textit{action} among all admissible sequences with the same endpoints. 
Equivalently, each finite portion of the orbit realizes the shortest possible polygonal path connecting its 
endpoints with reflections inside $\Omega$. 

A \emph{minimizing orbit} of the billiard map is an orbit whose impact sequence $(s_i)$ is a minimizing configuration 
for the action functional. Such orbits exist for every rotation number \cite{ForniMather, Gole} and provides the variational foundation for Mather's $\beta$-function. 

Given a minimizing orbit $\underline s = (s_k)$, the following limit
\[
\omega = \lim_k \frac{s_k}{k}
\]
is well-defined as a number of $[0,1]$ called rotation number of $\underline s$. Mather's $\beta$-function can be defined by the following formula

\begin{definition}
Mather's beta-function is given for any $\omega\in(0,1/2]$ by 
\begin{equation}
\beta(\omega) = -\lim_{N\to+\infty} \frac{1}{2N}\sum_{k=-N}^{N-1} |\gamma(s_{k+1}) - \gamma(s_k)|.
\end{equation}
\end{definition}

where the sequence $(s_k)_k$ corresponds to a minimizing orbit of rotation number $\omega$.
It is known \cite{FKS, Siburg} that $\beta$ is a strictly convex function of $\omega$. It is related to the so-called marked length-spectrum of a billiard. It has therefore properties related to isospectral deformations of domains which we describe below. 

Let $(\Omega_{\tau})_{\tau\in I}$ be a one-parameter family of domains. We can consider for any $\tau$ the Mather's $\beta$-function $\beta_{\tau}$ associated to the domain $\Omega_{\tau}$.

\begin{theorem}[\cite{FKS} or \cite{Siburg}]
\label{theorem:beta_isospectral_domains}
A $\mathscr C^3$-smooth deformation $(\Omega_{\tau})_{\tau\in I}$ is isospectral if and only if $\beta_{\tau}=\beta_0$ for any $\tau\in I$.
\end{theorem}

The idea of the proof relies on Sard's Lemma, from which it follows that the set $\mathscr L(\Omega_{\tau})$ is a set of zero measure, and hence minimizing periodic orbits of $\Omega_{\tau}$ must have their perimeter constant. Hence $\beta_{\tau}(\omega)$ is constant in $\tau$ for any rational rotation number $\omega$. Complete proofs for domains with $\mathscr C^{\infty}$-smooth boundary can be found in \cite[Theorem 2.12]{FKS} and in \cite[Corollary 3.2.3]{Siburg}. Below we describe how to adapt the proof to domains with $\mathscr C^3$-smooth boundary.

\begin{proof}
Sard's Lemma states that the set of critical values of a $\mathscr C^r$ real-valued function 
on a $d$-dimensional manifold has measure zero provided that $r \geq d$.  We need to apply  Sard's Lemma to the functionals $\mathcal{L}_q$ describing the lengths of polygons with $q$ sides, $q\geq 3$. Hence, there is no problem when  $r = \infty$, being these functionals $\mathscr C ^\infty$. 
However, if $r < \infty$, one should apply Sard's lemma in a more ingeneous way.
For any $q$, let $\Theta_q$ be the map that sends 
$(s_0, s_1)$ to the $q$-tuple $(s_0, s_1, \cdots, s_{q-1})$ obtained by collecting the next 
$q - 2$ collision points of the billiard trajectory passing through the initial pair 
$(s_0, s_1)$. Notice that the map $\Theta_q$ is as smooth as the billiard map, hence it is 
$\mathscr C^{r-1}$. One then considers
\[
\mathcal{L}^*_q(s_0, s_1) = \mathcal{L}_q(\Theta_q(s_0, s_1)).
\]
and observe that the set of critical values of $\mathcal{L}_q$ is the same as the set of critical values 
of $\mathcal{L}^*_q$; the latter is a real-valued $\mathscr C^{r-1}$ function of a 2-dimensional 
manifold; hence the result follows provided that $r \geq 3$.
\end{proof}

A \textit{rotational invariant curve} of rotation number $\omega\in\RR$ of a strongly convex planar domain $\Omega$ is a smooth non-contractible embedding 
\[
\SS^1\xhookrightarrow{} M
\]
such that $M\setminus\SS^1$ has two connected components and for which there is a parametrization
\[
\Gamma:\SS^1\longrightarrow M
\]
as smooth as the curve and such that
\begin{equation}
\label{equation:conj_inv_curve}
f\left(\Gamma(\theta)\right) = \Gamma(\theta+\omega)
\end{equation}
for any $\theta\in\SS^1$. If \eqref{equation:conj_inv_curve} is satisfied, we will say that $\Gamma$ conjugates $f$ to a rotation.

\begin{proposition}
\label{proposition:expression_beta_inv_curve}
Assume that the domain $\Omega$ whose boundary is parametrized by $\gamma$ has a rotational invariant curve of irrational rotation number $\omega\in(0,1/2]$ parametrized by
\[
\Gamma:\theta\in\SS^1\mapsto(s(\theta),\varphi(\theta))\in M,
\]
which conjugates $f$ to a rotation.

Then $\beta(\omega)$ can be expressed as
\begin{equation}
\label{equation:beta_inv_curve}
\beta(\omega) = -\int_0^1|\gamma(s(\theta+\omega))-\gamma(s(\theta))|d\theta.
\end{equation}
\end{proposition}

\begin{proof}
Orbits on an invariant curve are action minimizing, see for example \cite{FierobeSorrentino, ForniMather} hence 
\[
\beta(\omega) = -\lim_{N\to+\infty} \frac{1}{2N}\sum_{k=-N}^{N-1} |\gamma(s(\theta+\omega+k\omega)) - \gamma(s(\theta+k\omega))|.
\]
and the result follows from Birkhoff's ergodic theorem.
\end{proof}

\subsection{Deformation map}

Let $(\Omega_{\tau})_{\tau\in I}$ be a $\mathscr C^1$-smooth one-parameter family of domains. For any given $\tau\in I$, consider an $|\partial\Omega_{\tau}|$-periodic map $\gamma_{\tau}:\RR\to\RR^2$ parametrizing the boundary of $\partial\Omega_{\tau}$ by arc-length.

\begin{definition}
The \textit{deformation map} associated to $(\Omega_{\tau})_{\tau\in I}$ is the family of maps $(n_{\tau})_{\tau\in I}$ defined for any $\tau\in I$ and $s\in\RR$ by
\begin{equation} \label{defntau}
n_{\tau}(s) = \langle\partial_{\tau}\gamma_{\tau}(s)\,,\,N_{\gamma_{\tau}}(s)\rangle,
\end{equation}
where $\langle\,,\,\rangle$ denotes the canonical scalar product on $\RR^2$ and $N_{\gamma_{\tau}}(s)$ is the outgoing unit normal vector to $\partial\Omega_{\tau}$ at the point $\gamma_{\tau}(s)$. 
\end{definition}

\begin{remark}
\label{remark:symmetry_n}
The maps $n_{\tau}$ reflect the symmetries of a deformation:
\begin{itemize}
\item If the domains $\Omega_{\tau}$ are symmetric with respect to an axis independant of $\tau$ and $s=0$ corresponds to a point on that axis, then 
\[
n_{\tau}(-s)=n_{\tau}(s),\qquad s\in\RR,\tau\in I.
\]
\item If moreover each domain is centrally-symmetric, $n$ is $\frac{1}{2}|\partial \Omega_{\tau}|$-periodic:
\[
n_{\tau}\left(s+\frac{1}{2}|\partial \Omega_{\tau}|\right)=n_{\tau}(s),\qquad s\in\RR,\tau\in I.
\]
\end{itemize}
In what follows, given a deformation of dihedrally symmetric domains, we will assume that all domains in the deformation share the same symmetry axis and the same center of symmetry: this can be done without loss of generality as translations and rotations do not affect the length spectrum. We will therefore consider parametrizations by arc-length $s$ of the domains' boundaries such that the parameter $s=0$ corresponds to a point on the common axis of symmetry.
\end{remark}

A family of deformation maps characterizes trivial deformations in the following sense.

\begin{proposition}
\label{proposition:zero_deformation}
Assume that $n_{\tau}$ vanishes identically in $\tau$. Then $\Omega_{\tau}=\Omega_0$ for any $\tau\in I$.
\end{proposition}

\begin{proof}
The proof can be found in \cite[Lemma 3.3]{DKW}: if $n_{\tau}=0$ for any $\tau\in I$, this means that the vectors $\partial_{\tau}\gamma_{\tau}(s)$ and $\gamma_{\tau}'(s)$ are linearly dependent for any $s$. Since $\gamma(\tau,\cdot)$ is assumed to be a diffeomorphism, we conclude that $d\gamma(\tau, s)$ has rank
$1$ everywhere. It follows from the Constant Rank Theorem that the
image of $\gamma$ is a manifold of dimension $1$ which can be parameterized by
$\gamma(0,\cdot)$. This implies that $\Omega_\tau = \Omega_0$ for any $|\tau| \leq 1$.

\end{proof}

\subsection{$\tau$-variation of Mather's $\beta$-function}
\label{subsection:tau_variation_beta}

Let $(\Omega_{\tau})_{\tau\in I}$ be a one-parameter family of Birkhoff tables, associated to a deformation map $n_{\tau}(s)$ which is $|\partial\Omega_{\tau}|$-periodic in $s$. Denote by $\beta_{\tau}(\omega)$ Mather's $\beta$-function associated to each billiard map.

Consider an irrational number $\omega\in (0,1/2)$ and assume that for each $\tau\in I$ the billiard map associated to the table $\Omega_{\tau}$ has a rotational invariant curve $\Gamma_{\tau}$ of rotation number $\omega$, 
on which the dynamics is cojugated to a rotation, with conjugation given by: 
\[
\Gamma_{\tau}:\theta\in\SS^1\mapsto (s_{\tau}(\theta),\varphi_{\tau}(\theta)).
\]

 \medskip
 
\begin{proposition}
\label{proposition:first_var_beta}
Assume that each domain $\Omega_{\tau}$ has a rotational invariant curve parametrized by $\Gamma_{\tau}$ of irrational rotation number $\omega$ such that the map $(\tau,\theta)\mapsto\Gamma_{\tau}(\theta)$ is $\mathscr C^1$-smooth and conjugates the billiard map in $\Omega_{\tau}$ to a rotation. Then the map $\tau\mapsto\beta_{\tau}(\omega)$ is differentiable and its $\tau$-derivative is given by
\begin{equation}
\label{equation:beta_form1}
\partial_{\tau}\beta_{\tau}(\omega) = 2\int_0^{1}n_{\tau}(s_{\tau}(\theta))\sin\varphi_{\tau}(\theta)d\theta.
\end{equation}
\end{proposition}

\begin{proof}
According to \eqref{equation:beta_inv_curve}, $\beta_{\tau}(\omega)$ admits the following expression:
\[
\beta_{\tau}(\omega) = -\int_0^1|\gamma_{\tau}(s_{\tau,\omega}(\theta+\omega))-\gamma_{\tau}(s_{\tau,\omega}(\theta))|d\theta,
\]
where $s_{\tau,\omega}(\theta)$ is the $s$-projection of $\Gamma_{\tau}(\theta)$, and $\gamma_{\tau}$ is the corresponding parametrization of $\partial\Omega_{\tau}$ by $s$.

To ease the computations, we drop the index $\omega$ in $s_{\tau,\omega}$. Let us consider the familiy of generating maps $L_{\tau}$, namely the maps defined for any $\tau$ by
\[
L_{\tau}(s_0,s_1) := -|\gamma_{\tau}(s_1)-\gamma_{\tau}(s_0)|,\qquad s_0,s_1\in\RR.
\]
Then
\begin{equation}
\label{equation:expression_beta_proof}
\beta_{\tau}(\omega) = -\int_0^{1}L_{\tau}(\gamma_{\tau}(s_{\tau}(\theta)),\gamma_{\tau}(s_{\tau}(\theta+\omega)))d\theta
\end{equation}

Differentiating \eqref{equation:expression_beta_proof} with respect to $\tau$ we obtain the formula
\[
-\partial_{\tau}\beta_{\tau}(\omega) = A+B+C
\]
where
\[
A := \int_0^{1} \partial_{\tau} L_{\tau}(\gamma_{\tau}(s_{\tau}(\theta)),\gamma_{\tau}(s_{\tau}(\theta+\omega)))d\theta,
\]
\[
B := \int_0^{1}
\partial_{\tau}\gamma_{\tau}(s_{\tau}(\theta)) \partial_{s_0}L_{\tau}(\gamma_{\tau}(s_{\tau}(\theta)),\gamma_{\tau}(s_{\tau}(\theta+\omega)))d\theta,
\]
and 
\[
C := \int_0^{1}
\partial_{\tau}\gamma_{\tau}(s_{\tau}(\theta+\omega)) \partial_{s_1}L_{\tau}(\gamma_{\tau}(s_{\tau}(\theta)),\gamma_{\tau}(s_{\tau}(\theta+\omega)))d\theta.
\]
Applying the change of coordinates $\theta'=\theta+\omega$ in the integral defining $C$, we obtain
\begin{eqnarray*}
C+B &=& \int_0^{1}
\partial_{\tau}\gamma_{\tau}(s_{\tau}(\theta))
\left( 
\partial_{s_1}L_{\tau}(\gamma_{\tau}(s_{\tau}(\theta-\omega)),\gamma_{\tau}(s_{\tau}(\theta)))\right. \\
&& \left. \qquad + \;
\partial_{s_0}L_{\tau}(\gamma_{\tau}(s_{\tau}(\theta)),\gamma_{\tau}(s_{\tau}(\theta+\omega)))
\right)
d\theta. 
\end{eqnarray*}
Since $L_\tau$ is a generating function for the billiard dynamics, then
\[
\partial_{s_1}L_{\tau}(\gamma_{\tau}(s_{\tau}(\theta-\omega)),\gamma_{\tau}(s_{\tau}(\theta)))
+
\partial_{s_0}L_{\tau}(\gamma_{\tau}(s_{\tau}(\theta)),\gamma_{\tau}(s_{\tau}(\theta+\omega))) \equiv 0.
\]
Hence, $B+C=0$. 

Now, if we set $s_0 = s_{\tau}(\theta)$, $s_1 = s_{\tau}(\theta+\omega)$, then
\[
\partial_{\tau} L_{\tau}(s_0,s_1) = 
\langle \partial_{\tau}\gamma_{\tau}(s_0), u^{\tau}(\theta) \rangle
-\langle \partial_{\tau}\gamma_{\tau}(s_1), u^{\tau}(\theta) \rangle,
\]
where $u^{\tau}(\theta)$ is the unit vector joining $\gamma_{\tau}(s_0)$ to $\gamma_{\tau}(s_1)$. Applying again the change of coordinates $\theta'=\theta+\omega$ while integrating the second term of $\partial_{\tau} L_{\tau}(s_{\tau}(\theta),s_{\tau}(\theta+\omega))$, we obtain
\[
\partial_{\tau}\beta_{\tau}(\omega) = \int_{0}^{1} 
\langle \partial_{\tau}\gamma_{\tau}(s_{\tau}(\theta)), u^{\tau}(\theta)-u^{\tau}(\theta-\omega) \rangle d\theta.
\]
Because of the billiard reflexion law off the boundary $\partial\Omega_{\tau}$ at the point $\gamma_{\tau}(s_{\tau}(\theta))$, 
\[
u^{\tau}(\theta)-u^{\tau}(\theta-\omega) = -2\sin\varphi_{\tau}(\theta)N_{\gamma_{\tau}}(s_{\tau}(\theta))
\]
and the claim then follows by using the definition of $n_\tau$ in \eqref{defntau}.
\end{proof}

\medskip
For each $\tau\in I$ suppose that we are given a diffeomorphism
\[
L_{\tau}:x\in[0,1]\to s\in[0,|\partial\Omega_{\tau}|].
\]
We write $n_{\tau}(x)$ for any $x\in[0,1]$ to refer to $n_{\tau}\circ L_{\tau}(x)$. Similarly, $\varphi_{\tau}(x)$ refers to the angle $\varphi$ such that the pair $(s=L_{\tau}(x),\varphi)$ belong to the invariant curve parametrized by $\Gamma_{\tau}$. The latter notations are justified by the fact that $\Gamma_{\tau}(\SS^1)$ is a graph over the $s$-coordinate, as it follows from Mather's graph theorem  and the fact that all orbits on it are minimizing (see \cite{ForniMather}).\\

We will assume that $L_{\tau}$ preserves the symmetries of the domains, see Remark \ref{remark:symmetry_n}: if $\Omega_{\tau}$ is axis-symmetric, then we assume that $L_{\tau}(-x)=-L_{\tau}(x)$; if moreover $\Omega_{\tau}$ is centrally-symmetric, then $L_{\tau}\left(x+\frac{1}{2}\right)=L_{\tau}(x)+\frac{1}{2}|\partial\Omega_{\tau}|$.\\

\begin{corollary}
\label{corollary:first_var_beta_density}
Given a collection of diffeomorphisms $L_{\tau}:x\in[0,1]\to s\in[0,|\partial\Omega_{\tau}|]$, the $\tau$-derivative of $\beta_{\tau}(\omega)$ is given by
\begin{equation}
\label{equation:beta_form2}
\partial_{\tau}\beta_{\tau}(\omega) = \int_0^{1}n(x)\mu_{\Omega_{\tau}}(\omega,x)dx,
\end{equation}
where $\mu_{\Omega_{\tau}}(\omega,x)$ is given by the expression
\[
\mu_{\Omega_{\tau}}(\omega,x) = 2\sin\varphi_{\tau}(x)\theta_{\tau}'(x)
\]
in which $\theta_{\tau}$ is the change of coordinate from $x$ to $\theta$, and can be expressed as $\theta_{\tau}(x) = s_{\tau}^{-1}\circ L_{\tau}(x)$.
\end{corollary}

\medskip

In this paper, we will consider $x$ to be the {\it Lazutkin coordinate} on $\partial\Omega_{\tau}$. It was introduced by Lazutkin \cite{Lazutkin} as part of a set of coordinates $(x,y)$ on the phase space defined explicitely by
\begin{equation}
\label{equation:Lazutkin_change}
x = L^{-1}_{\tau}(s) = C_{\tau}\int_0^s \varrho_{\tau}^{-2/3}(s')ds'
\qquad\text{and}\qquad 
y=4C\varrho_{\tau}^{1/3}(s)\sin\left(\frac{\varphi}{2}\right).
\end{equation}
where $\varrho_{\tau}(s)$ denotes  the radius of curvature of $\partial\Omega_{\tau}$ at the point $\gamma_\tau(s)$, while $C_{\tau}^{-1} := \int_0^{|\partial\Omega_{\tau}|} \varrho_{\tau}^{-2/3}(s)ds$ is a normalization constant.\\

\begin{definition}
\label{definition:KAM_density}
We call the map $\mu_{\Omega_\tau}(\omega,\cdot)$ the \textit{KAM density} of the domain $\Omega_\tau$. 
\end{definition}

\subsection{KAM properties of a billiard table}

In the above-defined Lazutkin's coordinates $(x,y)$ defined by \eqref{equation:Lazutkin_change},
the billiard  $f(x,y)=(x_1,y_1)$ becomes
\[
\left\{
\begin{array}{ccl}
x_1 & = & x+y+\mathscr O(y^3)\\
y_1 & = & y +\mathscr O(y^4).
\end{array}\right.
\]
This expansion allows one to interpret $f$ as a perturbation for small $y$ of the integrable standard map whose phase space is foliated by invariant curves. As a consequence of KAM theorem, rotational invariant curves corresponding to small Diophantine numbers exist in regions of the phase space nearby $\{y=0\}$. More precisely, they exist in abundance, namely for a large (in measure sense) family of rotation number.
Let us first recall these results.

Given $(\nu,\sigma)\in(0,1)\times(\frac{5}{2},+\infty)$, we define the set of \textit{$(\nu,\sigma)$-Diophantine numbers} by
\[
\mathcal D(\nu,\sigma) := \{ \omega\in (0,1/2) \,|\, \forall (m,n)\in\ZZ\times\ZZ_{>0}
\quad |n\omega-m|\geq \nu |m| n^{-\sigma}\}.
\]

Let $\Omega$ be a Birkhoff table with $\mathscr C^r$-smooth boundary and $f:M\to M$ the associated billiard map inside $\Omega$. 

Lazutkin \cite{Lazutkin} showed that every $\mathscr C^r$-smooth Birkhoff table $\Omega$ with $r\geq 6$ sufficiently large admits smooth KAM curve $\Gamma$ for a positive measure set of Diophantine rotation numbers $\omega$ that accumulates to $0$. In particular, these curves will be $\mathscr C^\ell$, where $\ell=\ell(r, \omega)$. 

More precisely, given a fixed pair $(\nu,\sigma)$ there is $\delta = \delta(r, \nu,\sigma)>0$ such that $\Omega$ admits smooth KAM curves $\Gamma$ of any rotation number $\omega\in \mathcal D(\nu,\sigma)\cap(0,\delta)$, whose regularity $\ell=\ell(r)$ can be chosen uniformly for Diophantine $\omega$ of the same type $(\nu,\sigma)$. Moreover, $\ell(r)\to+\infty$ as $r\to+\infty$.

%
%
%

Moreover it follows from \cite{Douady} that  if two $\mathscr C^r$-smooth Birkhoff tables $\Omega$ and $\Omega'$ are sufficiently close, with respect to the $\mathscr C^r$-norm, then their associated billiard maps admits KAM curves corresponding to rotation numbers in  a common positive measure set of Diophantine numbers that accumulate to $0$. 

\begin{definition}
\label{definition:KAM_family_curves}
Let $\Omega$ be a $\mathscr C^r$-smooth billiard domain whose billiard map is given by $f$. A \textit{family of KAM invariant curve of type $(\nu,\sigma)$} is a map
\[
\Gamma: E\times\SS^1\to M
\]
where $E$ is a set of the form $\mathcal D(\nu,\sigma)\cap(0,\delta)$ for a $\delta>0$ and such that for any $\omega\in E$ the map $\Gamma(\omega,\cdot)$ is a non contractible embedding $\SS^1\longrightarrow M$ which satisfies
\[
f\left(\Gamma(\omega,\theta)\right)=\Gamma(\omega,\theta+\omega),\qquad\theta\in\SS^1.
\]

It is said to be $\mathscr C^{\ell}$-smooth for a certain integer $\ell>0$ if $\Gamma$ is $\mathscr C^{\ell}$-smooth in  Whitney sense.
\end{definition}

\begin{theorem}[\cite{Douady, Lazutkin}, \cite{Poeschel1,Poeschel2}]
\label{theorem:KAM_general}
Let $(\nu,\sigma)\in(0,1)\times(5/2,+\infty)$ and $r>0$. There exist $\delta=\delta(r,\nu,\sigma)>0$ and $\ell=\ell(r,\nu,\sigma)>0$ such that:
\begin{enumerate}
\item any strongly convex planar domain  $\Omega$ with $\mathscr C^r$-smooth boundary has a $\mathscr C^{\ell}$-smooth family of KAM curves $\Gamma$ of type $(\nu,\sigma)$;
\item if two such domains $\Omega$ and $\Omega'$ are sufficiently $\mathscr C^r$-close, they admit a family of KAM curves $\Gamma$ and $\Gamma'$ of type $(\nu,\sigma)$ defined on the same set $\mathcal D(\nu,\sigma)\cap(0,\delta)$;
\item $\Gamma$ and $\Gamma'$ can be taken sufficiently $\mathscr C^{\ell}$-close from eachother.
\end{enumerate}
Moreover, the integer $\ell(r,\nu,\sigma)$ is such that $\ell(r,\nu,\sigma)\to+\infty$ as $r\to+\infty$ and the pair $(\nu,\sigma)$ is fixed.
\end{theorem}

\begin{definition}
Given $\delta>0$, a \textit{KAM curve of type $(\nu,\sigma)$ in a $\delta$-neighborhood of the boundary} of $\Omega$ is a map
\[
\Gamma:\mathcal D(\nu,\sigma)\cap[0,\delta) \times\SS^1 \to M
\]
such that $\Gamma(\omega,\cdot)$ parametrizes a rotational invariant curve of rotation number $\omega$ for any $\omega\in\mathcal D(\nu,\sigma)$  and conjugates $f$ to a rotation.
\end{definition}

Given an integer $\ell>0$, we say that a KAM curve $\Gamma:\mathcal D(\nu,\sigma)\cap[0,\delta) \times\SS^1 \to M$ is {$\mathscr C^{\ell}$-smooth} if it is the restriction of a $\mathscr C^{\ell}$ smooth maps defined on $[0,\delta) \times\SS^1$ -- this is known as Whitney smooth regularity. 

This induces a topology on the space of KAM curve of type $(\nu,\sigma)$ defined in a $\delta$-neighborhood of the boundary. However given two such curves $\Gamma,\Gamma'$ which are close from eachother, the $\ell$-jets in $\omega$ of $\Gamma$ and $\Gamma'$ at an $\omega \in \mathcal D(\nu,\sigma)\cap[0,\delta)$ are not necessarily close to eachother, unless $\omega$ is a \textit{density point} of $\mathcal D(\nu,\sigma)\cap[0,\delta)$.

\bigskip

\subsection{Isospectral orthogonality}
\label{subsection:isospectral_orthogonality}

Let $(\Omega_{\tau})_{\tau\in I}$ be a $\mathscr C^3$-smooth one-parameter family of Birkhoff tables obtained by a deformation map $n_{\tau}$. Assume that each domain $\Omega_{\tau}$ has a rotational invariant curve parametrized by $\Gamma_{\tau}$ of rotation number $\omega$ such that the map $(\tau,\theta)\mapsto\Gamma_{\tau}(\theta)$ is $\mathscr C^1$-smooth  and conjugates the billiard map inside $\Omega_{\tau}$ to a rotation.

\begin{proposition}
\label{proposition:othogonality_property}
Assume that the family $(\Omega_{\tau})_{\tau\in I}$ is isospectral. Then for any $\tau\in I$ the following orthogonality property holds:
\begin{equation}
\label{equation:isospectral_orthonality2}
\int_0^{1}n_{\tau}(x)\mu_{\Omega_{\tau}}(\omega,x)dx = 0.
\end{equation}
\end{proposition}

\begin{proof}
By Theorem \ref{theorem:beta_isospectral_domains}, $\tau\mapsto\beta_{\tau}(\omega)$ is constant. Hence its $\tau$ derivative (see \eqref{equation:beta_form2}) vanishes identically. 
\end{proof}

Theorem \ref{theorem:KAM_general} imply that if we are given a $\mathscr C^{r}$-smooth family of strongly convex planar billiards $(\Omega_{\tau})_{\tau\in I}$ then $\Omega_{\tau}$ has a $\mathscr C^{\ell}$-smooth KAM curve $\Gamma_{\tau}$ of type $(\nu,\sigma)$ for any $\tau\in I$ such that the map 
\[
(\tau,\omega,\theta)\mapsto \Gamma_{\tau}(\omega,\theta)
\]
is $\mathscr C^{\ell}$-smooth. Proposition \ref{proposition:othogonality_property} has the following consequence:

\begin{corollary}
\label{corollary:orthogonality_property}
Let 
$(\nu,\sigma)\in(0,1)\times(1,+\infty)$, 
$r,\ell>0$, $\delta,\varepsilon>0$ and $\Omega$ be as in Theorem \ref{theorem:KAM_general}. Let $(\Omega_{\tau})_{\tau\in I}$ be an isospectral $\mathscr C^r$-smooth deformation such that $\|\partial\Omega_{\tau}-\partial\Omega\|_{\mathscr C^r}\leq \varepsilon$ for any $\tau\in I$.
Then
\begin{equation}
\label{equation:isospectral_orthonality}
\int_0^{1}n_{\tau}(x)\mu_{\Omega_{\tau}}(\omega,x)dx = 0\qquad \forall\omega\in\mathcal D(\nu,\sigma),\,\,\forall\tau\in I.
\end{equation}
\end{corollary}

\medskip

\subsection{Periodic orbits of rotation number $1/q$}
\label{subsection:DKW_proof}
Let $(\Omega_{\tau})_{\tau\in I}$ be a one-parameter family of Birkhoff tables and assume that for any $\tau\in I$ the domain $\Omega_{\tau}$ is axis-symmetric. This case was studied in \cite{DKW}, and we recall some results that will be useful for our discussion.

For each $\tau\in I$, we can assume that $\Omega_{\tau}$ share the same symmetry axis as $\Omega_0$, and that a same point, called \textit{marked point}, independent of $\tau$ belong to this axis and to the boundary $\partial\Omega_{\tau}$: this is justified by eventually applying translations and rotations to the different domains in the family, since these transformations do not change their length-spectrum.

In terms of parametrization by arc-length, we can assume that $\gamma_{\tau}(0)$ is constant in $\tau$, equal to the marked point, and that for any $\tau\in I$ and any $s\in\RR$, the points $\gamma_{\tau}(s)$ and $\gamma_{\tau}(-s)$ are symmetric with respect to the previously fixed axis of symmetry.

Given such a symmetric domain $\Omega$, the lift of its billiard map $F:\RR\times(0,\pi)\to \RR\times(0,\pi)$ and an integer $q\geq 2$, we consider symmetric periodic orbits of rotation number $1/q$, namely a sequence $(s_k,\varphi_k)_{k\in\ZZ}$ where $F^k(s_0,\varphi_0)=(s_k,\varphi_k)$ for any $k\in\ZZ$ and such that for $k\in\ZZ$
\[
(s_{k+q},\varphi_{k+q}) = (s_k+1,\varphi_k),\quad
s_{-k} = -s_k.
\]

\medskip
\begin{proposition}[\cite{DKW}]
\label{proposition:isospectral_DKW}
Given an axis-symmetric domain $\Omega$, for each $q\geq 2$ the billiard map in $\Omega$ has a ``distinguished'' symmetric periodic orbit $(s_k^{(q)},\varphi_k^{(q)})_k$ of rotation number $1/q$ starting at $s_0^{(q)}=0$.  

Moreover if the deformation $(\Omega_{\tau})_{\tau\in I}$ is isospectral and $\mathscr C^9$-smooth then
\begin{equation} 
\ell_q(n_{\tau}) := \sum_{k=0}^{q-1}n_{\tau}(s_k^{(q,\tau)})\sin\varphi_k^{(q,\tau)} = 0,\quad\forall q\geq 2
\end{equation}
where $(s_k^{(q,\tau)},\varphi_k^{(q,\tau)})_k$ is the distinguished symmetric periodic orbit  of rotation number $1/q$ in $\Omega_{\tau}$.
\end{proposition}

Let $\Omega$ be an axis-symmetric domain with $\mathscr C^9$-smooth boundary. Note that if a deformation is isospectral, each domain has the same perimeter, and -- up to applying a rescaling --  we can suppose that the common perimeter is $1$. For any $q\geq 2$, consider its distinguished symmetric periodic orbit of rotation number $1/q$ from Proposition \ref{proposition:isospectral_DKW}; 
in the expression of these periodic orbits, we will use Lazutkin coordinate $x$ on the boundary instead of the arc-length parameter and denote by $(x_k^{(q)},\varphi_k^{(q)})_k$ these periodic orbits.

Given a map $n\in L^2(\SS^1)$, define
\[
\ell_q(n) = \sum_{k=0}^{q-1}n(x_k^{(q)})\sin\varphi_k^{(q)}.
\]

Fix an $\alpha\in(3,4)$ and recall that an even map $n\in L^2(\SS^1)$ can be decomposed in Fourier modes as
\[
n(x) = \sum_{j\geq 0}\widehat n_j \cos(2\pi jx).
\]
We introduce the space $H^{\alpha}$ of even maps $n\in L^2(\SS^1)$ such that $j^{\alpha}\widehat n_j$ converges to $0$ in $j$. We endow it with the norm $\|\cdot\|_{\alpha}$ defined by
\[
\|n\|_{\alpha} = \sup_{j\geq 0} j^{\alpha}|\widehat n_{j}|.
\]

\begin{remark}
Note that here the $\tfrac{1}{2}$-periodicity is not assumed.
\end{remark}

We now present a result on asymptotics estimates of $\ell_q(n)$, as $q\to+\infty$, which can be found in \cite{DKW}. To simplify the statement, we introduce the operator 
$\Delta:H^{\alpha}\to h^{\alpha}$ defined for $q\geq 0$ by
\[
\Delta(n)_q = \tfrac{1}{q}\sum_{k=0}^{q-1}n\left(\tfrac{k}{q}\right)-\widehat n_0.
\]
Analogously, we consider the space $h^{\alpha}$ of sequences $(u_q)_{q\geq 2}$ such that $q^{\alpha}u_q$ converges to $0$ in $q$ together  with the norm $\|\cdot\|_{\alpha}$ defined by $\|u\|_{\alpha} = \sum q^{\alpha}|u_q|$.

\begin{proposition}[\cite{DKW}]
\label{proposition:asymptotics_DKW}
Given $n\in H^{\alpha}$, there exist linear maps $\ell_0, \ell_{\bullet}:H^\alpha\to\RR$ such that the following expansion holds:
\begin{equation}
\label{equation:asymptotics_ell}
\ell_q(n) = \Delta(m_{\Omega}n)_q+\ell_0(m_{\Omega} n)+\frac{1}{q^2}\ell_{\bullet}(m_{\Omega} n) +\bigo{\frac{\|n\|_{\alpha}}{q^4}}
\end{equation}
where $m_{\Omega}$ is the map given by
\begin{equation}
\label{equation:asymptotics_ell_2}
m_{\Omega}(x) = \left(2C\varrho_{\Omega}(x)\right)^{-1}>0
\end{equation}
in which $C>0$ corresponds to the normalization constant appearing in~\eqref{equation:Lazutkin_change} and $\varrho_{\Omega}$ corresponds to the radius of curvature of $\partial\Omega$ expressed in terms of the Lazutkin coordinates $x$.
\end{proposition}

This proposition suggests to introduce a renormalized version of previous objects, namely 
\[
\tilde n := \frac{n}{m_{\Omega}},
\quad
\tilde \ell_q(\tilde n) = \ell_q(m_{\Omega}\tilde n),\qquad \tilde n\in H^\alpha,\, q\geq 2.
\]
From now on we will drop the tilde above $n$ and work with maps $n\in H^{\alpha}$. So we write $\tilde\ell_q(n)$ for
\[
\tilde \ell_q(n) = \ell_q(m_{\Omega}n).
\]

We introduce the subspace $H^{\alpha}_{1/2}\subset H^{\alpha}$ of even $\tfrac{1}{2}$-periodic maps, \textit{i.e.}, the space of maps $n\in H^{\alpha}$ satisfying for any $x\in\RR$ the relations
\[
n(-x) = n(x),\quad n(x+\tfrac{1}{2}) = n(x).
\]
Note that a map $n\in H^{\alpha}$ belongs to $H^{\alpha}_{1/2}$ if and only if $\widehat n_{2j+1}=0$ for any $j\geq 0$, and therefore the sequence of Fourier coefficients of $n$ is given by $(\widehat n_{2j})_{j\geq 0}$. The norm $\|\cdot\|_{\alpha}$ restricts naturally to this space.

For a given $q_0>0$ we introduce the space 
\[
H^{\alpha}_{1/2,q_0} = \{n\in H^{\alpha}_{1/2}\,|\, \forall j< q_0\quad \widehat n_{2j}=0\}.
\]
We endow it with the norm induced by $\|\cdot\|_{\alpha}$. We define analogously the space $h^{\alpha}_{q_0}$ of sequences $u=(u_j)_{j\geq q_0}$ such that $j^{\alpha}\widehat u_{j}$ converges to $0$ with $j$.

Consider the operator $S_{\Omega}^{q_0}:H^{\alpha}_{1/2}\to h^{\alpha}_{q_0}$ defined for any $n\in H^{\alpha}_{1/2}$ by
\[
S_{\Omega}^{q_0}(n)_q = \tilde \ell_{2q}(n) - \ell_0(n)-\frac{1}{(2q)^2}\ell_{\bullet}(n),\qquad q\geq q_0.
\]
The operator $S_{\Omega}^{q_0}$ satisfies the following result:
\begin{proposition} {If a $\mathscr C^9$-smooth} deformation $(\Omega_{\tau})_{\tau\in I}$ is isospectral then
\begin{equation}
S_{\Omega_{\tau}}^{q_0}\left(m_{\Omega_{\tau}}n_{\tau}\right)_q = 0,\quad\forall q\geq q_0,\,\tau\in I.
\end{equation}
\end{proposition}

\begin{proof}
Fix a given $\tau\in I$. Since the deformation is isospectral, $\ell_q(n_{\tau}) = 0$ for any $q\geq 2$, see Proposition \ref{proposition:isospectral_DKW}. From the asymptotic expansion of $\ell_q$ given in \eqref{equation:asymptotics_ell}, it follows that 
\[
\ell_0(m_{\Omega_{\tau}}n_{\tau})=\ell_{\bullet}(m_{\Omega_{\tau}}n_{\tau})=0.
\]
This comes from the expansion \eqref{equation:asymptotics_ell} combined with two observations: 

1) for $s\in\RR$, if one consider the sequence $u^{(s)} = (q^{s})_{q>0}$, then the family $(u^{(s)}) _{s\in\RR}$ is linearly independent in the space of sequences.

2) since the deformation is $\mathscr C^9$-smooth, the map $m_{\Omega_{\tau}}n_{\tau}$ is in $H^{\alpha}$ and therefore $\Delta(m_{\Omega_{\tau}}n_{\tau})_q=\mathcal{O}(q^{-\alpha})$, see Proposition \ref{proposition:invertibility_dirichlet_mobius}. 

The result follows.
\end{proof}

Let $D_{\Omega}^{q_0}:H^{\alpha}_{1/2,q_0}\to h^{\alpha}_{q_0}$ be the restriction of $S_{\Omega}^{q_0}$ to $H^{\alpha}_{1/2,q_0}$.

\begin{proposition}[\cite{DKW}]
\label{proposition:inveritbility_DKW}
If $\Omega$ is a Birkhoff table with a $\mathscr C^9$-smooth axis-symmetric boundary, then there exists $q_0\geq 2$ such that $D_{\Omega}^{q_0}:H^{\alpha}_{1/2,q_0}\to h^{\alpha}_{q_0}$ is invertible.
\end{proposition}

\begin{proof}
Let $q_0\geq 2$. By construction of $\tilde\ell_q$ for $q\geq q_0$, we can decompose $D_{\Omega}^{q_0}$ in 
\[
D_{\Omega}^{q_0} = \Delta_s+R
\]
where $\Delta_s$ is the symmetric Dirichlet operator defined by
\[
\Delta_s(n)_q = \Delta(n)_{2q},\qquad n\in H^{\alpha}_{1/2},\, q\geq q_0,
\]
and $R:H^{\alpha}_{1/2,q_0}\to h^{\alpha}_{q_0}$ is a bounded operator whose norm satisfies
\[
\|R\|\leq \frac{K}{q_0^{4-\alpha}}
\]
for a given constant $K>0$ independant of $q_0$.

By Proposition \ref{proposition:invertibility_dirichlet_mobius}, $\Delta_s$ is invertible with inverse the operator $M$ defined in Definition \ref{definition:dirichlet_mobius}, which we call \textit{Möbius} operator. Hence
\[
D_{\Omega}^{q_0} = \Delta_s(I+MR)
\]
where $I$ is the invertible map associating to any $n\in H^{\alpha}_{1/2,q_0}$ the sequence of its Fourier coefficients. Norm estimates give
\[
\|MR\|_{\alpha}\leq \|M\|_{\alpha}\|R\|_{\alpha}\leq \frac{KK^{\ast}}{q_0^{4-\alpha}},
\]
where $K^{\ast}>0$ is a constant independent of $q_0$ which bounds the norm of $M$ -- it exists by Proposition \ref{proposition:invertibility_dirichlet_mobius}. 

Therefore $\|MR\|_{\alpha}<1$ for sufficiently large $q_0$ and hence $D_{\Omega}^{q_0}$ is invertible in this case.
\end{proof}

\begin{corollary}[\cite{DKW}]
\label{proposition:dimension_DKW}
Let $q_0\geq 2$ such that $D_{\Omega}^{q_0}$ is invertible. The operator $S_{\Omega}^{q_0}$ is onto and its kernel has dimension $q_0$.
\end{corollary}

\begin{proof}
The operator $S_{\Omega}^{q_0}$ is surjective since it restricts to $H^{\alpha}_{1/2,q_0}$ as an invertible operator.

Now, let $n\in H^{\alpha}_{1/2}$. Write $n = n_L+n_H$ where $n_L(x):=\sum_{j=0}^{q_0-1}\widehat n_j\cos(2\pi jx)$ and $n_H=n-n_L$. This decomposition and the invertibility of $D_{\Omega}^{q_0}$ induces the equivalence
\[
S_{\Omega}^{q_0}(n) = 0
\quad
\Leftrightarrow
\quad
n_H = -(D_{\Omega}^{q_0})^{-1}S_{\Omega}^{q_0}(n_L)
\]
and the result follows.
\end{proof}

\section{Isospectral operator}
\label{section:isospectral_operator}
Let an integer $q_0\geq 2$ and $(\nu,\sigma)\in(0,1)\times(\frac{5}{2},+\infty)$. Theorem \ref{theorem:KAM_general} provides a $\delta=\delta(r,\nu,\sigma)>0$ and an integer $\ell = \ell(r,\nu,\sigma)>0$ for any integer $r$ sufficiently large. Since $\ell(r,\nu,\sigma)\to+\infty$ as $r\to+\infty$, we can consider an integer $r\geq 9$ sufficiently large such that $\ell\geq q_0$.

As a consequence of Theorem \ref{theorem:KAM_general}, any
strongly convex planar billiard domain $\Omega$ with $\mathscr C^r$-smooth boundary has a family of KAM curve of type $(\nu,\sigma)$ defined for rotation numbers in $\mathcal D(\nu,\sigma)\cap (0,\delta)$.

Let $\Omega$ be a domain with $\mathscr C^r$-smooth boundary. We can consider its KAM density $\mu_{\Omega}$ -- see Definition \ref{definition:KAM_density} -- which is therefore $\mathscr C^{\ell}$-smooth.

If $\omega_0\in(0,1/2)$ is a density point of $\mathcal D(\nu,\sigma)$ and $J=(j_0,\ldots,j_{q_0-1})$ with $j_i \in\{0,\ldots,N\}$, we consider the map
\[
T_{\Omega}^{J,\omega_0}: H^{\alpha}_{1/2} \to h^{\alpha}
\]
defined for all $n\in H^{\alpha}_{1/2}$ by
\begin{equation}
T_{\Omega}^{J,\omega_0}(n)_q = \left\{
\begin{array}{cc}\vspace*{0.2cm}
\int_0^{1} \frac{n(x)}{m_{\Omega}(x)}\partial_{\omega}^{j_q}\mu_{\Omega}(\omega_0,x)dx & \text{if }q<q_0\\
S_{\Omega}^{q_0}(n)_q & \text{if }q\geq q_0.
\end{array}
\right.
\end{equation}

As $\Delta_s(H^{\alpha}_{1/2})\subset h^{\alpha}$, see Definition \ref{definition:dirichlet_mobius}, it follows from Proposition \ref{proposition:asymptotics_DKW} that

\begin{proposition}
\label{proposition:operators_isospectral_deformations}
$T_{\Omega}^{J,\omega_0}$ defines a bounded operator 
\[
T_{\Omega}^{J,\omega_0}:H^{\alpha}_{1/2} \to h^{\alpha}.
\]
If a $\mathscr C^r$-smooth deformation $(\Omega_{\tau})_{\tau\in I}$ is isospectral and if we set $\Omega_0=\Omega$ and $n := n_0$ then
\[
T_{\Omega}^{J,\omega_0}(m_{\Omega} n) = 0.
\]
\end{proposition}

Hence we are brought to study, the injectivity of $T_{\Omega}^{J,\omega_0}$. Since injectivity of operators on infinite dimensional spaces is not stable under small perturbations, we will study invertibility properties of $T_{\Omega}^{J,\omega_0}$.

\begin{proposition}
\label{proposition:operators_continuity}
Given $\delta>0$, there exists $\varepsilon>0$ and an integer $r>0$ such that if $\Omega'$ is $\varepsilon$-$\mathscr C^r$-close to $\Omega$, the operator $T_{\Omega'}^{J,\omega_0}$ is well-defined and
\[
\|T_{\Omega}^{J,\omega_0}-T_{\Omega'}^{J,\omega_0}\|_{\alpha}<\delta.
\] 
\end{proposition}

\begin{proof}
The continuity of $\Omega'\mapsto S_{\Omega'}^{q_0}$ in the topology of domains with $\mathscr C^9$-smooth boundary follows from \cite{DKW}. 

For the continuity of the $q_0$ first values of $T_{\Omega}^{J,\omega_0}$, we apply Theorem \ref{theorem:KAM_general}: there exists $r>0$ such that if $\Omega'$ is sufficiently $\mathscr C^r$-close to $\Omega$, it has a KAM curve $\Gamma_{\Omega'}$ of type $(\nu,\sigma)$ defined in the same $\delta$-neighbohood of the boundary, hence containing $\omega_0$ as a density point in its set of rotation numbers. Moreover the map $\Omega'\mapsto\Gamma_{\Omega'}$ is continuous from the space of domains with $\mathscr C^r$-smooth boundary to the space of $\mathscr C^{q_0}$-smooth functions. Hence so does the map $\Omega'\mapsto\mu_{\Omega'}$, which concludes the result.
\end{proof}

\begin{theorem}
\label{theorem:invertibility_ellipse}
Let $\Omega=\mathscr E$ be an ellipse which is not a disk. There exists $q_0>0$ such that for any $(\nu,\sigma)\in(0,1)\times(\frac{5}{2},+\infty)$ and any accumulation point $\omega_0\in\mathcal D(\nu,\sigma)$ there exists $J=(j_0,\ldots,j_{q_0})$ with $j_0<\ldots<j_{q_0-1}$ for which the operator
$T_{\mathscr E}^{J,\omega_0}$ is invertible. 
\end{theorem}

\begin{proof}
Let $\Omega=\mathscr E$ be an ellipse which is not a disk. Choose $q_0\geq 2$ such that the operator $D_{\mathscr E}^{q_0}:H^{\alpha}_{q_0}\to h^{\alpha}_{q_0}$ associated to $\mathscr E$ and introduced in Subsection \ref{subsection:DKW_proof} is invertible. Recall that we are then given a surjective operator $S=S_{\mathscr E}^{q_0}$ whose kernel has dimension $q_0$ -- see Proposition \ref{proposition:dimension_DKW}.

The strategy is to \textit{complete} $S$ by linear forms $f_0,\ldots,f_{q_0-1}$ in the sense of Definition \ref{definition:completion} to build an operator $T=S_f$ so that the assumptions of Proposition \ref{proposition:invertibility_completion} are satisfied: the invertibility of $T$ will follow then directly.

Fix $(\nu,\sigma)\in(0,1)\times(\frac{5}{2},+\infty)$ and an accumulation point $\omega_0\in\mathcal D(\nu,\sigma)$. For $j\geq 0$, consider the linear map $f_j:L^2_{1/2}(\SS^1)\to\RR$ defined by
\[
f_j(n)= \int_0^{1} \frac{n(x)}{m_{\Omega}(x)}\partial_{\omega}^{j}\mu_{\Omega}(\omega_0,x)dx,
\qquad n\in L^2_{1/2}(\SS^1).
\]

By Proposition \ref{proposition:totality_ellipse},
\[
\cap_{j\geq 0}\ker f_j = \{0\}.
\]

Now we observe that given a finite dimensional space $V\subset L^2_{1/2}(\SS^1)$ of dimension $d>0$ and $j\geq0$, the intersection $V\cap \ker f_j$ is either $V$ or has dimension $d-1$. 

Since $\ker S$ has dimension $q_0$, we can construct inductively $q_0$ integers $j_0<\ldots<j_{q_0-1}$ such that 
\[
\dim \left(\ker S\cap\left(\cap_{p=0}^q \ker f_{j_p}\right)\right) = q_0-1-q,\qquad 0\leq q<q_0.
\]
This implies by construction that 
\[
\ker S\cap\left(\cap_{q=0}^{q_0-1} \ker f_{j_q}\right) = \{0\}.
\]

Therefore the completion $T = S_f$ of $S$ by $f$ -- see Definition \ref{definition:completion} -- satisfies the assumptions of Proposition \ref{proposition:invertibility_completion} and the result is proven.
\end{proof}

\section{Proof of Theorem \ref{theorem:main}}
\label{section:proof_main}

Let an ellipse $\mathscr E$ which is not a disk, and fix $\alpha\in(3,4)$. By Theorem \ref{theorem:invertibility_ellipse}, there exists $q_0>0$ and $J=(j_0,\ldots,j_{q_0})$ with $j_0<\ldots<j_{q_0-1}$ such that for $\omega_0=0$ the operator
\[
T_{\mathscr E}^{J,\omega_0}:H^{\alpha}_{1/2} \to h^{\alpha}
\] 
is invertible. 

Since the set of bounded invertible operators between Banach spaces is an open set, there is a $\delta>0$ such that any operator $T':H^{\alpha}_{1/2} \to h^{\alpha}$ satisfying $\|T'-T_{\mathscr E}^{J,\omega_0}\|_{\alpha}<\delta$ is also invertible.

But by Proposition \ref{proposition:operators_continuity}, one can find $\varepsilon>0$ and an integer $r>0$ such that for any strongly planar convex domain $\Omega$ with $\mathscr C^r$-smooth boundary which is $\varepsilon$-$\mathscr C^r$-close to $\mathscr E$, the corresponding operator 
$T_{\Omega}^{J,\omega_0}$
is well-defined
and satisfies
\[
\|T_{\Omega}^{J,\omega_0}-T_{\mathscr E}^{J,\omega_0}\|_{\alpha}<\delta.
\]
It is in particular invertible. In particular its kernel is trivial by Proposition \ref{proposition:operators_isospectral_deformations} and the proof is complete.

\appendix
\section{The family $\{\sin^{2j} \varphi \}_{j\geq 0}$}
\label{section:sinus_family}

This section is devoted to the study of the family of functions $\{\sin^{2j} \varphi \}_{j\geq 0}$. They appear in the expression of the KAM density of the ellipse, and play an important role in the proof of the key proposition of this paper, namely Proposition \ref{proposition:totality_ellipse}. More precisely, we show that they are trigonometric polynomial whose coefficients can be computed explicitely.

\begin{proposition}
\label{proposition:even_powers_sin}
Let $j\geq 0$. The function $\sin^{2j} \varphi$ admits the following Fourier expansion:
\[
\sin^{2j}\varphi = \sum_{k=0}^j s_{jk} \cos(2k\varphi)
\]
where
\[
s_{jk} := \left\{
\begin{array}{cl}
\frac{1}{4^j}\binom{2j}{j}&\text{if } k=0\\
2\frac{(-1)^{k}}{4^j}\binom{2j}{j-k}&\text{if } k>0.
\end{array}
\right.
\]
\end{proposition}

\begin{proof}
The proof relies on the following binomial expansion:
\[
\sin^{2j}\varphi = \left(\frac{e^{i\varphi}-e^{-i\varphi}}{2i}\right)^{2j}
=\frac{(-1)^j}{4^j}\sum_{k=0}^{2j}\binom{2j}{k}(-1)^ke^{-2i\varphi(j-k)}.
\]
In this sum, there is one constant term corresponding to $k=j$ and inducing the formula for $s_{j0}$. For $k\neq j$, the two terms corresponding to $k$ and $2j-k$ sum up to $s_{jk}\cos(2k\varphi)$.
\end{proof}

\medskip

\begin{proposition}
\label{proposition:sin_fourier_vanish}
Let $N\in L^2(\SS^1)$ be an even and $1/2$-periodic map.  If $N$ satisfies
\begin{equation}
\label{equation:sin_fourier_vanish}
 \int_0^{2\pi}N(\varphi) \sin^{2j}\varphi d\varphi = 0 \qquad \forall\; j\geq 0,
\end{equation}
then $N=0$ almost everywhere.
\end{proposition}

\begin{proof}
By Proposition \ref{proposition:even_powers_sin} and the triangular structure of the coefficients $s_{jk}$, Condition \eqref{equation:sin_fourier_vanish} is equivalent to say the even Fourier coefficients of $N$ vanish. Since $N$ is even and $1/2$-periodic, $N$ has to be the zero map (up to a set of zero measure).
\end{proof}

\section{Möbius and Dirichlet operators}
\label{section:mobius_operator}

Let $q_0\geq 2$ be an integer and $\alpha\in(3,4)$.

\begin{definition}
\label{definition:dirichlet_mobius}
The \textit{symmetric Dirichlet operator} is the map $\Delta_s:H^{\alpha}_{1/2,q_0}\to h^{\alpha}_{q_0}$ defined for all $n\in H^{\alpha}_{1/2,q_0}$ by
\[
\Delta_s(n)_q := \tfrac{1}{2q}\sum_{k=0}^{2q-1}n\left(\tfrac{k}{2q}\right)-\widehat n_0
= \sum_{p>0}\widehat n_{2p q}.
\]

The \textit{Möbius operator} is the map $M:h^{\alpha}_{q_0}\to H^{\alpha}_{1/2,q_0}$ defined for all $u\in h^{\alpha}_{1/2,q_0}$ by $n=M(u)$ where $n\in H^{\alpha}_{1/2,q_0}$ is the function whose even Fourier coefficients are given by
\[
\widehat n_{2j} = \sum_{\ell>0}\mathscr M(\ell)u_{2\ell j}
\]
where $\mathscr M$ is the Möbius function defined according to the formula
\[
\mathscr M(\ell) = \left\{
\begin{array}{cl}
1 & \text{if } \ell=1;\\
(-1)^k & \text{if } \ell \text{ is the product of }k\text{ distinct primes;}\\
0 & \text{otherwise.}
\end{array}
\right.
\]

\end{definition}

\begin{proposition}
\label{proposition:invertibility_dirichlet_mobius}
Let $q_0\geq 2$. The operators $\Delta_s$ and $M$ are invertible bounded operators, which are one the inverse of the other. Moreover the norm of $M$ is uniformly bounded in $q_0$.
\end{proposition}

\begin{proof}
Given $n\in H^{\alpha}_{1/2,q_0}$ and $q\geq q_0$, 
\[
q^{\alpha}|\Delta_s(n)_q|
\leq q^{\alpha} \sum_{p>0}|\widehat n_{2p q}|
\leq  \sum_{p>0}(2p q)^{\alpha}|\widehat n_{2p q}|\cdot \frac{1}{(2p)^{\alpha}}.
\]
But note, by definition of the norm on $H^{\alpha}_{1/2}$, that
\[
(2p q)^{\alpha}|\widehat n_{2p q}| \leq \|n\|_{\alpha},\qquad q\geq2.
\]
Hence
\[
\|\Delta_s(n)\|_{\alpha} \leq \sum_{p>0}\|n\|_{\alpha} \frac{1}{(2p)^{\alpha}}
=  \frac{\zeta(\alpha)}{2^{\alpha}}\|n\|_{\alpha},
\]
where
\[
\zeta(\alpha) = \sum_{p>0}\frac{1}{p^{\alpha}}.
\]

Similar computations gives the same result for $M$, since $|\mathscr M|\leq 1$, namely  
\[
\|M(n)\|_{\alpha} \leq \frac{\zeta(\alpha)}{2^{\alpha}}\|n\|_{\alpha}.
\]
Note that this bound is independent of $q_0$.

Now given $n\in H^{\alpha}_{1/2,q_0}$, consider the map $N = M(\Delta_s(n))$. By definition of Möbius operator $M$, for $j\geq q_0$, 
\[
\widehat N_{2j} = \sum_{\ell>0}\mathscr M(\ell)\Delta_s(n)_{2\ell j}
= \sum_{\ell>0}\sum_{p>0}\mathscr M(\ell)\widehat n_{2p\ell j}
= \sum_{k>0}\left(\sum_{\ell|k}\mathscr M(\ell)\right) \widehat n_{2kj}.
\]
However, for an integer $k>0$ the value of $\sum_{\ell|k}\mathscr M(\ell)$ is zero except when $k=1$, and in this case it is $1$. Hence $\widehat N_{2j} = \widehat n_{2j}$ and $N=n$. This completes the proof.
\end{proof}

\section{Ellipses are total}
\label{section:ellipses_total}

In this section, we compute the KAM density $\mu_{\mathscr E}(\omega,x)$ of an ellipse $\mathscr E$ which is not a disk, and we show the following result:

\begin{proposition}
\label{proposition:totality_ellipse}
Let $\mathscr E$ be an ellipse which is not a disk. Then, the associated KAM density $\mu_{\varepsilon}$ is defined as a real analytic function of $(\omega,\theta)$
\[
\mu_{\mathscr E}:[0,\tfrac{1}{2})\times \SS^1\to\RR.
\]
For any given $\omega_0\in[0,\tfrac{1}{2})$, the family of partial derivatives
\[
\left(\partial_{\omega}^{j}\mu_{\mathscr E}(\omega_0,\cdot)\right)_{j\geq 0}
\]
is a total set of $L^2_{1/2}(\SS^1)$, namely it satisfies
\begin{equation}
\label{equation:total_family}
\bigcap_{j\geq 0}\left(\partial_{\omega}^{j}\mu_{\mathscr E}(\omega_0,\cdot)\right)^{\perp} = \{0\}.
\end{equation}
\end{proposition}

\medskip

\begin{remark}
In fact Formula \eqref{equation:total_family} is equivalent to say that the family
\[
\left(\partial_{\omega}^{j}\mu_{\mathscr E}(\omega_0,\cdot)\right)_{j\geq 0}
\]
spans a dense subspace of $L^2_{1/2}(\SS^1)$. 
\end{remark}

This section will be devoted to the proof of Proposition \ref{proposition:totality_ellipse}. {In Subsection \ref{subsection:KAM_density_ellipse} we give an explicit expression of $\mu_{\mathscr E}$ and in Subsection \ref{subsection:proof_totality} we prove Proposition \ref{proposition:totality_ellipse}.}

\subsection{KAM density of an ellipse}
\label{subsection:KAM_density_ellipse}

Given $a>b>0$, consider the ellipse $\mathscr E$ described by the pairs $(X,Y)\in\RR^2$ satisfying the Equation
\[
\mathscr E:\quad \frac{X^2}{a^2}+\frac{Y^2}{b^2}=1.
\]
Its eccentricity is given by $e = \sqrt{1-(b/a)^2}\in[0,1)$. With this definition, $\mathscr E$ is a disk if and only if $e=0$.

Consider the family of confocal ellipses $C_\lambda$ {contained in $\mathscr E$} given by the equations
\[
C_{\lambda}:\quad \frac{X^2}{a^2-\lambda^2}+\frac{Y^2}{b^2-\lambda^2}=1,
\qquad\lambda\in[0,b)
\]
and whose eccentricities are given by
\[
k_{\lambda} = \sqrt{\frac{a^2-b^2}{a^2-\lambda^2}}\in [e,1),
\qquad\lambda\in [0,b).
\]

It is known \cite{FKS} that $C_{\lambda}$, for $\lambda\in[0,b)$, corresponds to a \textit{caustic} of the billiard in $\mathscr E$, which means that is a segment of the  billiard trajectory  is tangent to $C_{\lambda}$, then the trajectory will remain tangent to it after every reflections, both in the future and in the past. In the case of $\lambda=0$, we can consider the ellipse itself as a caustic consisting of fixed points of the billiard map inside of the ellipse. This translates into the existence of a rotational invariant curve for the billiard map in $\mathscr E$. Given $\lambda\in[0,b)$, the invariant curve corresponding to $C_{\lambda}$ has a rotation number $\omega(\lambda)$ such that the correspondance
\[
\lambda\in[0,b)\mapsto\omega(\lambda)\in[0,1/2)
\]
is an analytic diffeomorphism \cite{KS}. To simplify, we will write $\lambda_{\omega}$ for the inverse, and $k_{\omega}$ for $k_{\lambda_{\omega}}$.

Consider the parametrization of $\mathscr E$ by $\gamma$ where
\[
\gamma(\phi) := (a\cos(\phi),b\sin(\phi)),\qquad\phi\in[0,2\pi).
\]
This induces a reparametrization of the boundary of the ellipse by a coordinate $\phi\in[0,2\pi)$. We compute first the KAM density $\tilde\mu_{\mathscr E}(\omega,\phi)$ of $\mathscr E$ with respect to this coordinate $\phi$:

\begin{proposition}
\label{proposition:beta_ellipse}
Let $(\nu,\sigma)\in(0,1)\times(\frac{5}{2},+\infty)$.
The KAM density $\tilde\mu_{\mathscr E}(\omega,\phi)$ of the ellipse $\mathscr E$ with respect to the coordinate $\phi$  is given for any $(\omega,\phi)\in\mathcal D(\nu,\sigma)\times[0,2\pi)$, where $\omega$ is an accumulation point, by
\begin{equation}
\label{equation:beta_ellipse}
\tilde\mu_{\mathscr E}(\omega,\phi) = 
\frac{\pi\lambda_{\omega}}{K(k_{\omega})}\cdot
\frac{1}
{\sqrt{(a^2\sin^2\phi+b^2\cos^2\phi)(1-k_{\omega}^2\sin^2\phi)}}
\end{equation}
where
\[
K(k) := \int_{0}^{\pi/2}\frac{1}{\sqrt{1-k^2\sin^2\phi}}d\phi
\]
denotes the complete elliptic integral of the first kind.
\end{proposition}

\begin{proof}[Proof of Proposition \ref{proposition:beta_ellipse}]
We apply Proposition \ref{proposition:first_var_beta}, and we first do a change of coordinates 
$\theta = f_{\lambda}(\varphi)$ where the function $f_{\lambda}$ has the explicit expression 
\[
\theta := f_{\lambda}(\phi) = \frac{\pi}{2}\frac{F(\phi,k_{\lambda})}{K(k_{\lambda})}
\]
where
\[
F(\phi,k_{\lambda}) = \int_{0}^{\phi}\frac{1}{\sqrt{1-k_{\lambda}^2\sin^2\phi'}}d\phi'.
\]

Doing a change of coordinates $\theta\mapsto\phi$ in Equation \eqref{equation:beta_form1}, the density becomes
\begin{equation}
\mu_{\mathscr E}(\lambda,\phi) = 2\sin\varphi_{\lambda}(\phi)f_{\lambda}'(\phi)
=\frac{\pi}{K(k_{\lambda})}\cdot \frac{\sin\varphi_{\lambda}(\phi)}{\sqrt{1-k_{\lambda}^2\sin^2\phi}}d\varphi
\end{equation}
where $\varphi_{\lambda}(\phi)$ stands for the angle of the reflection at the point of parameter $\phi$.

Now there is a relation between $\sin\varphi_{\lambda}(\phi)$ and the so-called {\it Joachimstall invariant} \cite{FierobeCMPLX, Tabach} $J_{\lambda}$ given by
\[
J_{\lambda} := \frac{Xv_X}{a^2}+\frac{Yv_Y}{b^2}
\]
where $(X,Y)=(a\cos\phi,b\sin\phi)$ corresponds to a point  on the boundary of the ellipse and $v=(v_X,v_Y)$ is a vector 
parallel to the segment started from that point and tangent to the caustic of parameter $\lambda$ in the positive direction. The relation is as follows. Let 
\[
N = \left(\frac{X}{a^2},\frac{Y}{b^2}\right).
\]

Then $N$ is an outward normal to the boundary at $(X,Y)$ since it is proportional to the gradient of the Cartesian equation defining the ellipse. Previous definition of $J_{\lambda}$ can be expressed in the Euclidean norm as
\[
J_{\lambda} = -\sin\varphi_{\lambda}(\phi)\|N\|\|v\|.
\]
Now, as it was proven in \cite{FierobeCMPLX},
\[
\lambda = -\frac{abJ_{\lambda}}{\|v\|}
\]
which implies that
\[
\sin\varphi_{\lambda}(\phi) = \frac{\lambda}{ab\|N\|} = \frac{\lambda}{\sqrt{b^2\cos^2\phi+a^2\sin^2\phi}}
\]
and the result follows.
\end{proof}

Now we give the expression of the KAM density $\mu_{\mathscr E}(\omega,x)$ of the ellipse $\mathscr E$ in Lazutkin $x$-coordinate:

\begin{proposition}
\label{proposition:beta_ellipse_x}
The KAM density $\tilde\mu_{\mathscr E}(\omega,x)$ of the ellipse $\mathscr E$ in Lazutkin $x$-coordinate is given for any $(\omega,x)\in[0,1/2)\times[0,1)$ \textcolor{red}{(also $\omega =0$?) }by
\begin{equation}
\label{equation:beta_ellipse_x}
\mu_{\mathscr E}(\omega,x) = 
\frac{\pi\lambda_{\omega}}{K(k_{\omega})}\cdot
\frac{\phi'(x)}
{\sqrt{(a^2\sin^2\phi(x)+b^2\cos^2\phi(x))(1-k_{\omega}^2\sin^2\phi(x))}}
\end{equation}
where $\phi:[0,1)\to[0,2\pi)$ is the change of coordinates on the boundary from $x$ to $\phi$, see \eqref{equation:Lazutkin_change}.
\end{proposition}

\begin{remark}
For a disk, $\phi$ is given in terms of $x$ by $\phi_0(x) = 2\pi x$. For a general ellipse of eccentricity $e$, $\phi(x) = \phi_e(x)$ admits an explicit expression in integral form. We will not need it here, only the fact that $\phi_e\to\phi_0$ in the $\mathscr C^1$-topology as $e\to 0$.
\end{remark}

\begin{proof}
A change of coordinate from $x$ to $\phi$ in  \eqref{equation:beta_form2} implies that
\[
\mu_{\mathscr E}(\omega,x)= \tilde\mu_{\mathscr E}(\omega,\phi(x))\phi'(x)
\] 
and the result follows from Proposition \ref{proposition:beta_ellipse}.
\end{proof}

\subsection{Proof of Proposition \ref{proposition:totality_ellipse}}
\label{subsection:proof_totality}

Let $\mathscr E$ be an ellipse which is not a disk. The analiticity of $\mu_{\mathscr E}$ can be deduced from Equation \textcolor{red}{(which?)} and the analyticity of $\omega\mapsto k_{\omega}$.

Let $n\in L^2(\SS^1)$ and $\omega_0\in[0,1/2)$ be such that $\langle n\,|\,\partial_{\omega}^j\mu_{\mathscr E}(\omega_0,\cdot)\rangle = 0$ for any integer $j\geq 0$. This implies that
\[
\forall j\geq 0,\qquad \partial_{\omega}^j\left(\langle n\,|\,\mu_{\mathscr E}(\omega_0,\cdot)\rangle\right) = 0
\]
and therefore the infinite jet at $\omega_0$ of the analytic function
\[
\omega\in[0,1/2)\mapsto \langle n\,|\,\mu_{\mathscr E}(\omega_0,\cdot)\rangle
\]
vanishes. Hence the corresponding map vanishes identically in $\omega$:
\[
\langle n\,|\,\mu_{\mathscr E}(\omega,\cdot)\rangle = 0 \qquad \forall\,\omega\in[0,\tfrac{1}{2})
\]
which by doing a change of coordinates $\varphi = \phi(x)$ simplifies as
\begin{equation}
\label{equation:zero_density}
\forall\omega\in[0,\tfrac{1}{2})\qquad
\int_0^{2\pi}
\frac{n(\phi^{-1}(\varphi))} 
{\sqrt{(a^2\sin^2\varphi+b^2\cos^2\varphi)(1-k_{\omega}^2\sin^2\varphi)}}d\varphi =0.
\end{equation}
Now consider the Taylor expansion of the map $x\mapsto(1-x)^{-1/2}$ at $x=0$: 
\[
\frac{1}{\sqrt{1-x}} = \sum_{j\geq 0} c_j x^j
\]
where the coefficients $c_j$ are strictly positive and defined by the formula
\[
c_j = \frac{1}{4^j}\binom{2j}{j},\qquad j\geq 0.
\]
We apply this expansion in  \eqref{equation:zero_density}: if we denote by $N(\varphi)$ the map 
\[
N(\varphi) = \frac{n(\phi^{-1}(\varphi))} 
{\sqrt{a^2\sin^2\varphi+b^2\cos^2\varphi}}
\]

then
\[
0 = \int_0^{2\pi}\frac{N(\varphi)} 
{\sqrt{1-k_{\omega}^2\sin^2\varphi}}d\varphi = \sum_{j\geq 0}c_jk_{\omega}^{2j} \int_0^{2\pi} N(\varphi)\sin^{2j}\varphi d\varphi.
\]
Since $\omega$ is not a disk, the image of the map $\omega\mapsto k_{\omega}$ contains a non empty interval
(if $\omega \in (0,1/2)$, then the image is $(e,1)$). 
Hence by analyticity of previous expansion in $k_{\omega}$, and the fact that $c_j\neq 0$ for all $j\geq 0$, we deduce that
\[
\forall j\geq 0\qquad \int_0^{2\pi} N(\varphi)\sin^{2j}\varphi d\varphi = 0.
\]
Proposition \ref{proposition:sin_fourier_vanish} implies that $N=0$ and hence $n=0$, which concludes the proof of the result.
\qed

\section{Completion of operators}
\label{section:operator_completion}

Let $\alpha\in(3,4)$. For a given $q_0>0$ we introduce the space $h^{\alpha}_{q_0}$ of sequences $u=(u_j)_{j\geq q_0}$ such that 
$j^{\alpha}\widehat u_{j}$ converges to $0$ {as $j\to +\infty$}. We endow it with the norm $\|\cdot\|_{\alpha}$ defined by
\[
\|u\|_{\alpha} = \sup_{j\geq q_0} j^{\alpha}|u_{j}|.
\]

\begin{definition}
\label{definition:completion}
Let an operator $S:H^{\alpha}_{1/2}\to h^{\alpha}_{q_0}$ and a familiy $f=(f_0,\ldots,f_{q_0-1})$ of bounded linear forms $f_0,\ldots,f_{q_0-1}: H^{\alpha}_{1/2}\to\RR$. The \textit{$f$-completion of $S$} is the operator 
\[
S_f:H^{\alpha}_{1/2}\to h^{\alpha}
\]
defined for all $q\geq 0$ by 
\begin{equation}
S_f(n)_q := \left\{
\begin{array}{cc}\vspace*{0.2cm}
f_q(n) & \text{if }q<q_0\\
S(n)_q & \text{if }q\geq q_0.
\end{array}
\right.
\end{equation}
\end{definition}

\begin{proposition}[Invertibility of completions]
\label{proposition:invertibility_completion}
Assume that 
\begin{enumerate}
\item $S:H^{\alpha}_{1/2}\to h^{\alpha}_{q_0}$ is a bounded surjective operator;
\item $\ker S$ has dimension $q_0$;
\item $\ker f\cap\ker S=\{0\}$, where $\ker f = \cap_{q=0}^{q_0-1} \ker f_q$.
\end{enumerate}
Then, the $f$-completion $S_f:H^{\alpha}_{1/2}\to h^{\alpha}$ of $S$ is an invertible operator.
\end{proposition}

\begin{proof}
By construction $S_f:H^{\alpha}_{1/2}\to h^{\alpha}$ is a bounded operator. It is injective since if $\ker S_f = \ker S\cap\ker f = \{0\}$ by assumption. It remains to prove that $S_f$ is onto. We first show that $H^{\alpha}_{1/2}$ admits the decomposition 
\[
H^{\alpha}_{1/2} = \ker f\oplus\ker S.
\]
The last assumption implies that $f_0,\ldots,f_{q_0-1}$ are linearly independant, thus $\ker f$ has codimension $q_0$. Hence there is a $q_0$ dimensional space $V$ such that
\[
H^{\alpha}_{1/2} = \ker f\oplus V.
\]
But since $\dim \ker S=q_0$ by assumption, we can take $V=\ker S$ and the decomposition follows. Let $u = (v,w)\in h^{\alpha}$ where $v\in\RR^{q_0}$ and $w\in h^{\alpha}_{q_0}$. Since by assumptions $S$ is onto, there is $n\in H^{\gamma}$ such that $S(n) = w$. So if we decompose $n$ as
\[
n = n_f+n_S,
\]
where $n_f\in\ker f$ and $n_S\in \ker S$, we deduce that also $S(n_f) = w$. Moreover the map 
\[
f : \ker S\to\RR^{q_0}
\]
defined for all $n$ by $f(n) = (f_0(n),\ldots,f_{q_0-1})$ is injective by the assumption $\ker f\cap\ker S=\{0\}$. Hence it is onto as $\dim\ker S=q_0$. Therefore there is $\overline n\in\ker S$ such that $f(\overline n) = v$. By theses choices of $n_f$ and $\overline n$, 
\[
S_f(\overline n+n_f) = (f(\overline n), S(n_f))=(v,w) = u.
\]
Hence $S_f$ is onto and this completes the proof.
\end{proof}

\end{document}